\documentclass[a4paper,oneside]{article}
\usepackage[utf8]{inputenc}
\usepackage[a4paper,margin=3cm]{geometry} 
\usepackage{amsmath}
\usepackage{amsthm}
\usepackage{amscd}
\usepackage{amssymb}
\usepackage{MnSymbol}
\usepackage{latexsym}
\usepackage{eucal}
\usepackage{dsfont}
\usepackage{mathtools}
\usepackage{enumitem}
\usepackage{verbatim}

\usepackage{pdflscape}

\usepackage[colorlinks,pdftex]{hyperref}
\hypersetup{
linkcolor=black,
citecolor=black,
pdftitle={Existence and estimates of moments for Levy-type processes},
pdfauthor={Franziska K\"uhn},
pdfkeywords={generalized moments, Levy-type processes, fractional moments},
}

\makeatletter
\renewcommand\section{\@startsection{section}{1}{0mm}{-1.5\baselineskip}{\baselineskip}{\normalsize\bfseries\sffamily}}
\renewcommand\subsection{\@startsection{subsection}{1}{0mm}{-\baselineskip}{\baselineskip}{\normalsize\bfseries\sffamily}}
\makeatother

\makeatletter
\def\@fnsymbol#1{\ensuremath{\ifcase#1\or *\or **\or \dagger\or \ddagger\or
   \mathsection\or \mathparagraph\or \|\or \dagger\dagger
   \or \ddagger\ddagger \else\@ctrerr\fi}}

\newlength{\preskip}
\newlength{\postskip}
\makeatother

\newtheoremstyle{theorem}{\preskip}{\postskip}{\itshape}{}{\bfseries}{}
{.5em}{\textbf{\thmname{#1}\thmnumber{ #2} (\thmnote{ #3})}}
\newtheoremstyle{definition}{\preskip}{\postskip}{\normalfont}{0pt}{\bfseries}{}{.5em}{}
\newtheoremstyle{remark}{\preskip}{\postskip}{\normalfont}{0pt}{\bfseries}{}{.5em}{}

\swapnumbers
\theoremstyle{theorem} \newtheorem{thm}{Theorem}[section]
\theoremstyle{theorem} \newtheorem{lem}[thm]{Lemma}
\theoremstyle{theorem} 
\theoremstyle{theorem} \newtheorem{kor}[thm]{Corollary}
\theoremstyle{definition} 
\theoremstyle{definition} 
\theoremstyle{definition} \newtheorem*{ack}{Acknowledgements}
\theoremstyle{remark} \newtheorem*{bem}{Remark}
\theoremstyle{remark} 
\theoremstyle{definition}  \newtheorem{bsp}[thm]{Example}
\theoremstyle{definition}  
\DeclareMathOperator \re {Re}

\DeclareMathOperator \tr {tr}

\DeclareMathOperator \spt {supp}

\newcommand{\I}{\mathds{1}}

\newcommand{\cadlag}{c\`adl\`ag }

\newcommand\mc[1] {\mathcal{#1}}
\newcommand\mbb[1] {\mathds{#1}}

\newcommand\T{\rule{0pt}{3.5ex}}       
\newcommand\B{\rule[-6ex]{0pt}{0pt}} 

\author{%
    Franziska K\"{u}hn\thanks{Institut f\"ur Mathematische Stochastik, Fachrichtung Mathematik, Technische Universit\"at Dresden, 01062 Dresden, Germany, \texttt{franziska.kuehn1@tu-dresden.de}} 
}

\title{Existence and estimates of moments for L\'evy-type processes}

\date{}

\begin{document}

\maketitle

\abstract{\noindent 
    In this paper, we establish the existence of moments and moment estimates for L\'evy-type processes. We discuss whether the existence of moments is a time dependent distributional property, give sufficient conditions for the existence of moments and prove estimates of fractional moments. Our results apply in particular to SDEs and stable-like processes.
\par\medskip

 \par \medskip


\section{Introduction} \label{s-intro}

For a L\'evy process $(X_t)_{t \geq 0}$ and a submultiplicative function $f \geq 0$ it is known \begin{enumerate}
	\item \dots that the existence of the generalized moment $\mbb{E} f(X_t)$ does not depend on time, i.\,e.\ $\mbb{E}f(X_{t_0})<\infty$ for some $t_0>0$ implies $\mbb{E}f(X_t)<\infty$ for all $t \geq 0$, see e.\,g.\ \cite[Theorem 25.18]{sato}.  
	\item \dots that the existence of moments can be characterized in terms of the L\'evy triplet, see e.\,g.\ \cite[Theorem 25.3]{sato}. 
	\item \dots what the small-time asymptotics of fractional moments $\mbb{E}(|X_t|^{\alpha})$, $\alpha>0$, looks like, cf. \cite{deng} and \cite{luschgy}.
\end{enumerate}
The first two problems are of fundamental interest; the asymptotics of fractional moments has turned out to be of importance in various parts of probability theory, e.\,g.\ to obtain Harnack inequalities \cite{deng} or to prove the existence of densities for solutions of stochastic differential equations \cite{fou}. Up to now, there is very little known about the answers for the larger class of L\'evy-type processes which includes, in particular, stable-like processes, affine processes and solutions of (L\'evy-driven) stochastic differential equations. The aim of this work is to extend results which are known for L\'evy processes from the L\'evy case to L\'evy-type processes. \par
In the last years, heat kernel estimates for  L\'evy(-type) processes have attracted a lot of attention. Let us point out that the results obtained here have several applications in this area. In a future work\footnote{Update: The paper \cite{ihke} is available now.}, we will show that any rich L\'evy-type process $(X_t)_{t \geq 0}$ with triplet $(b(x),Q(x),N(x,dy))$ satisfies the integrated heat kernel estimate \begin{equation}
	\frac{\mbb{P}^x(|X_t-x| \geq R)}{t} \xrightarrow[]{t \to 0} N(x,\{y \in \mbb{R}^d; |y| \geq R\}) \label{intro-eq1}
\end{equation}
for all $R>0$ such that $N(x,\{y \in \mbb{R}^d; |y|=R\})=0$. Combining this with the statements from Section~\ref{s-exi} gives the small-time asymptotics of $t^{-1} \mbb{E}^x f(X_t)$ for a large class of functions $f$; the functions need not to be bounded or differentiable. The corresponding results for L\'evy processes have been discussed by Jacod \cite{jac} and Figueroa-L\'opez \cite{fig}. As suggested in \cite{fig}, this gives the possibility to extend the generator of the process to a larger class of functions. Moreover, following a similar approach as Fournier and Printems \cite{fou}, the estimates of the fractional moments show the existence of ($L^2$-)densities for L\'evy-type processes with Hölder-continuous symbols. \par
The structure of this paper is as follows. In Section~\ref{s-def}, we introduce basic definitions and notation. The problems mentioned above will be answered in Sections~\ref{s-time}--\ref{s-frac}; starting with the question whether the existence of moments is a time dependent distributional property in Section~\ref{s-time}, we give sufficient conditions for the existence of moments in Section~\ref{s-exi} and finally present estimates of fractional moments in Section~\ref{s-frac}. In each of these sections, we give a brief overview on known results, state some generalizations and illustrate them with examples.

\section{Basic definitions and notation} \label{s-def}

Let $(\Omega,\mathcal{A},\mathbb{P})$ be a probability space. For a random variable $X$ on $(\Omega,\mc{A},\mbb{P})$ we denote by $\mbb{P}_X$ the distribution of $X$ with respect to $\mathbb{P}$. We say that two functions $f,g: \mbb{R}^d \to \mbb{R}$ are \emph{comparable} and write $f \asymp g$ if there exists a constant $c>0$ such that $c^{-1} f(x) \leq g(x) \leq c f(x)$ for all $x \in \mbb{R}^d$. Moreover, we denote by $\mathcal{B}_b(\mbb{R}^d)$ the space of all bounded Borel-measurable functions $u: \mbb{R}^d \to \mathbb{R}$ and by $C_c^2(\mbb{R}^d)$ the space of functions with compact support which are twice continuously differentiable. For $x \in \mbb{R}^d$ and $r>0$ we set $B(x,r) := \{y \in \mbb{R}^d; |y-x|<r\}$ and $B[x,r] := \{y \in \mbb{R}^d; |y-x| \leq r\}$. The $j$-th unit vector in $\mbb{R}^d$ is denoted by $e_j$ and $x \cdot y = \sum_{j=1}^n x_j y_j$ is the Euclidean scalar product. For a function $u: \mbb{R}^d \to \mbb{R}$ we denote by $\partial_{x_j}^k u(x)$ the $k$-th order partial derivative with respect to $x_j$ and by $\nabla^2 u$ the Hessian matrix.  The \emph{Fourier transform} of an integrable function $u:\mbb{R}^d \to \mbb{R}$ is defined as \begin{equation*}
	\hat{u}(\xi) := \frac{1}{(2\pi)^d} \int_{\mbb{R}^d} e^{-i \, x \cdot \xi} u(x) \, dx, \qquad \xi \in \mbb{R}^d.
\end{equation*}
We call a stochastic process $(L_t)_{t \geq 0}$ a ($d$-dimensional) \emph{L\'evy process} if $L_0=0$ almost surely, $(L_t)_{t \geq 0}$ has stationary and independent increments and $t \mapsto L_t(\omega)$ is \cadlag for almost all $\omega \in \Omega$. It is well-known, cf.\ \cite{sato}, that $(L_t)_{t \geq 0}$ can be uniquely characterized via its \emph{characteristic exponent}, \begin{equation*}
	\psi(\xi) = - i \, b \cdot \xi + \frac{1}{2} \xi \cdot Q \xi + \int_{\mbb{R}^d \backslash \{0\}} (1-e^{i \, y \cdot \xi}+ i \, y \cdot \xi \I_{(0,1]}(|y|)) \, \nu(dy), \qquad \xi \in \mbb{R}^d;
\end{equation*}
here, $b \in \mbb{R}^d$, $Q \in \mbb{R}^{d \times d}$ is a symmetric positive semidefinite matrix and $\nu$ is a measure on $(\mbb{R}^d \backslash \{0\}, \mc{B}(\mbb{R}^d \backslash \{0\}))$ such that $\int_{\mbb{R}^d \backslash \{0\}} (|y|^2 \wedge 1) \, \nu(dy)<\infty$. The triplet $(b,Q,\nu)$ is called \emph{L\'evy triplet}. Our standard reference for L\'evy processes is the monograph by Sato \cite{sato}. A stochastic process $(X_t)_{t \geq 0}$ is said to be a (rich) \emph{L\'evy-type process} (or (rich) \emph{Feller process}) if $(X_t)_{t \geq 0}$ is a Markov process whose associated semigroup is Feller on the space of continuous functions vanishing at infinity and the domain of the generator contains the compactly supported smooth functions $C_c^{\infty}(\mbb{R}^d)$; for further details we refer the reader to \cite{ltp}. A theorem due to Courr\`ege and Waldenfels, cf.\ \cite[Corollary 2.23]{ltp}, states that the generator $A$ restricted to $C_c^{\infty}(\mbb{R}^d)$ is a pseudo-differential operator of the form \begin{equation*}
	Au(x) = - \int_{\mbb{R}^d} e^{i \, x \cdot \xi} q(x,\xi) \hat{u}(\xi) \, d\xi, \qquad u \in C_c^{\infty}(\mbb{R}^d), 
\end{equation*}
where \begin{equation}
	q(x,\xi) = q(x,0) - i \, b(x) \cdot \xi + \frac{1}{2} \xi \cdot Q(x) \xi + \int_{\mbb{R}^d} (1-e^{i \, y \cdot \xi}+ i \, y \cdot \xi \I_{(0,1]}(|y|))) \, N(x,dy) \label{def-eq3}
\end{equation}
is the \emph{symbol}. For each fixed $x \in \mbb{R}^d$, $(b(x),Q(x),N(x,dy))$ is a L\'evy triplet. Throughout this work, we will assume that $q(x,0)=0$. Using well-known results from Fourier analysis, it is not difficult to see that \begin{equation*}
	Au(x) = b(x) \cdot \nabla u(x)+\frac{1}{2} \tr(Q(x) \cdot \nabla^2 u(x)) + \int_{\mbb{R}^d \backslash \{0\}} (u(x+y)-u(x)-\nabla u(x) \cdot y \I_{(0,1]}(|y|)) \, N(x,dy)
\end{equation*}
for any $u \in C_c^{\infty}(\mbb{R}^d)$, see e.\,g.\ \cite[Theorem 2.21]{ltp}. We write $(X_t)_{t \geq 0} \sim (b(x),Q(x),N(x,dy))$ to indicate that $(X_t)_{t \geq 0}$ is a L\'evy-type process with triplet $(b(x),Q(x),N(x,dy))$. The symbol of a L\'evy-type process is locally bounded, cf.\ \cite[Theorem~2.27(d)]{ltp}. A L\'evy-type process has \emph{bounded coefficients} if $|q(x,\xi)| \leq C (1+|\xi|^2)$ for some constant $C>0$ which does not depend on $x \in \mbb{R}^d$. By \cite[Lemma~6.2]{schnurr}, the following statements are equivalent for any compact set $K \subseteq \mbb{R}^d$: \begin{enumerate}
	\item $\sup_{x \in K} \sup_{|\xi| \leq 1} |q(x,\xi)| <\infty$,
	\item $\sup_{x \in K} |q(x,\xi)| \leq C_K (1+|\xi|^2)$ for all $\xi \in \mbb{R}^d$,
	\item $\sup_{x \in K} (|b(x)|+|Q(x)|+ \int_{\mbb{R}^d \backslash \{0\}} (|y|^2 \wedge 1) \, N(x,dy)) < \infty$; here $|\cdot|$ denotes an arbitrary vector norm and matrix norm, respectively. 
\end{enumerate}
If $(X_t)_{t \geq 0}$ has bounded coefficients, then the statements are also equivalent for $K=\mbb{R}^d$. We will use the following result frequently; it is compiled from \cite[Theorem 3.13]{cinlar}. We remind the reader that a \emph{Cauchy process} is a L\'evy process with characteristic exponent $\psi(\xi) = |\xi|$. 

\begin{thm} \label{def-5}
	Let $(X_t)_{t \geq 0}$ be a L\'evy-type process with triplet $(b(x),Q(x),N(x,dy))$. There exist a Markov extension $(\Omega^{\circ}, \mathcal{A}^{\circ},\mathcal{F}_t^{\circ},\mbb{P}^{\circ,x})$, a Brownian motion $(W_t^{\circ})_{t \geq 0}$ and a Cauchy process $(L_t^{\circ})_{t \geq 0}$ with jump measure $N^{\circ}$ on $(\Omega^{\circ},\mathcal{A}^{\circ},\mc{F}_t^{\circ},\mbb{P}^{\circ,x})$ such that 
		\begin{equation*}
			X_t - X_0 
			= X_t^1 + X_t^2
		\end{equation*}
		with \begin{align*}
			X_t^1 &:= \int_0^t b(X_{s-}) \, ds + \int_0^t \sigma(X_{s-}) \, dW_s^{\circ} + \int_0^t \!\! \int_{|k| \leq 1} k(X_{s-},z) \, (N^{\circ}(dz,ds)-\nu^{\circ}(dz) \, ds) \\
			X_t^2 &:= \int_0^t \!\! \int_{|k|>1} k(X_{s-},z) \, N^{\circ}(dz,ds)
		\end{align*}
		for measurable functions $\sigma: \mbb{R}^d \to \mbb{R}^{d \times d}$ and $k: \mbb{R}^d \times (\mbb{R} \backslash \{0\}) \to \mbb{R}^d$ satisfying\begin{align}
			N(x,B) &= \int_{\mbb{R} \backslash \{0\}} \I_B(k(x,z)) \, \nu^{\circ}(dz), \qquad B \in \mathcal{B}(\mbb{R}^d \backslash \{0\}), \, x \in \mbb{R}^d, \label{def-eq20}
		\end{align}
		and $Q(x) = \sigma(x) \sigma(x)^T$; here $\nu^{\circ}(dz) = (2\pi)^{-1} z^{-2} \, dz$ denotes the L\'evy measure of a (one-dimensional) Cauchy process.
	\end{thm}

\section{Existence of moments - time independence} \label{s-time}

In this section we adress the question whether the existence of moments is a time dependent distributional property in the class of L\'evy-type processes. Given a L\'evy-type process $(X_t)_{t \geq 0}$ and a measurable function $f: \mathbb{R}^d \to [0,\infty)$, then under which additional assumptions on $(X_t)_{t \geq 0}$ and $f$ does the equivalence \begin{equation}
	\mathbb{E}^x f(X_t) < \infty \, \, \text{for some $t>0$} \iff \mathbb{E}^x f(X_t) < \infty \, \, \text{for all $t>0$}  \label{time-eq1}
\end{equation}
hold true? It is well-known that \eqref{time-eq1} holds for any L\'evy process $(X_t)_{t \geq 0}$ if $f$ is a locally bounded function which is submultiplicative (i.\,e.\ there exists $c>0$ such that $f(x+y) \leq c f(x) f(y)$ for all $x,y \in \mbb{R}^d$), see \cite[Theorem 25.3]{sato}. Analogous results for L\'evy-type processes seem to be unknown. First we discuss whether moments exist backward in time, i.\,e.\ whether \begin{equation}
	\mathbb{E}^x f(X_t) < \infty \, \, \text{for some $t>0$} \iff \mathbb{E}^x f(X_s) < \infty \, \, \text{for all $s \leq t$.}  \label{time-eq2}
\end{equation}
 The following theorem is the main result of this section.
 
\begin{thm} \label{time-1} 
	Let $(X_t)_{t \geq 0}$ be a L\'evy-type process with bounded coefficients and $f: \mbb{R}^d \to (0,\infty)$ measurable. \begin{enumerate}
		\item \label{time-1-i} Suppose there exists a bounded measurable function $g: \mbb{R}^d \to [0,\infty)$, such that $\inf_{|y| \leq r} g(y)>0$ for $r>0$ sufficiently small and \begin{equation}
			\inf_{y \in \mbb{R}^d} \frac{f(z+y)}{f(y)} \geq g(z) \label{time-eq4}
		\end{equation}
		for all $z \in \mbb{R}^d$. Then \begin{equation}
			\mbb{E}^{x} f(X_t) < \infty \iff \sup_{s \leq t} \mbb{E}^{x} f(X_s)< \infty. \label{time-eq3}
		\end{equation}
		\item\label{time-1-ii} \eqref{time-eq4}, hence \eqref{time-eq3}, holds if one of the following conditions is satisfied. \begin{enumerate}
			\item\label{time-1-ii-a} $f$ is submultiplicative and locally bounded.
			\item\label{time-1-ii-b} $\log f$ is H\"{o}lder continuous.
			\item\label{time-1-ii-c} $f$ is Hölder continuous and $\inf_{x \in \mbb{R}^d} f(x)>0$. 
			\item\label{time-1-ii-d} $f$ is differentiable and $\sup_{y \in \mbb{R}^d} \sup_{|z| \leq r} \tfrac{|\nabla f(y+z)|}{f(y)} <\infty$ for $r>0$ sufficiently small.
			\item\label{time-1-ii-e} $f$ is differentiable, $\inf_{y \in \mbb{R}^d} f(y)>0$, $\sup_{y \in \mbb{R}^d} \tfrac{|\nabla f(y)|}{f(y)} < \infty$ and $\nabla f$ is uniformly continuous.
		\end{enumerate}
	\end{enumerate}
\end{thm}

For the proof of Theorem~\ref{time-1} we need two auxiliary results. 

\begin{lem}[Maximal inequality] \label{frac-9}
	Let $(X_t)_{t \geq 0}$ be a L\'evy-type process with symbol $q$ and denote by $\tau_r^x := \inf\{t>0; X_t \notin B[x,r]\}$ the exit time from the closed ball $B[x,r]=\{y \in \mbb{R}^d; |y-x|<r\}$. Then there exists $C>0$ such that \begin{equation}
		\mbb{P}^x \left( \sup_{s \leq \sigma} |X_s-x| > r\right) \leq C \mbb{E}^x \left(\int_{[0,\sigma \wedge \tau_r^x)} \sup_{|\xi| \leq r^{-1}} |q(X_s,\xi)| \, ds \right) \label{frac-eq9}
	\end{equation}
	for all stopping times $\sigma$ and $r>0$. In particular, \begin{align}
		\mbb{P}^x \left( \sup_{s \leq \sigma} |X_s-x| > r \right) &\leq C \mbb{E}^x(\sigma) \sup_{|y-x| \leq r} \sup_{|\xi| \leq r^{-1}} |q(y,\xi)|. \label{frac-eq11}
	\end{align}
\end{lem}

Let us remark that \eqref{frac-eq11} is already known for $\sigma := t$, see \cite[Theorem 5.1]{ltp} for a proof.

\begin{proof}
	By the truncation inequality, see e.\,g.\ \cite[(Proof of) Lemma~1.6.2]{sasvari}, we have \begin{align*}
		\mbb{P}^x \left( \sup_{s \leq  \sigma} |X_s-x| > r\right) 
		\leq \mbb{P}^x(\tau_r^x \leq \sigma) 
		&\leq \mbb{P}^x(|X_{\sigma \wedge \tau_r^x}-x| \geq r) \\
		&\leq 7r^d \int_{[-r^{-1},r^{-1}]^d} \re (1-\mbb{E}^xe^{i \, \xi (X_{\sigma \wedge \tau_r^x}-x)}) \, d\xi.
	\end{align*}
	An application of Dynkin's formula yields \begin{align*}
		\mbb{P}^x \left( \sup_{s \leq t} |X_s-x| > r\right) 
		&\leq 7r^d \int_{[-r^{-1},r^{-1}]^d}  \re \mbb{E}^x \left( \int_{[0,\sigma \wedge \tau_r^x)} q(X_s,\xi) e^{i \, \xi (X_s-x)} \, ds \right) \, d\xi.
	\end{align*}
	Now \eqref{frac-eq9} follows from the triangle inequality and Fubini's theorem; \eqref{frac-eq11} is a direct consequence of \eqref{frac-eq9}.
\end{proof}

\begin{lem} \label{time-3}
	Let $(X_t)_{t \geq 0}$ be a L\'evy-type process with bounded coefficients and $g \in \mc{B}_b(\mbb{R}^d)$, $g \geq 0$, such that $\inf_{y \in B[0,r]} g(y)>0$ for $r>0$ sufficiently small. Then \begin{equation*}
		\exists \alpha>0,\delta>0 \, \, \forall x \in\mbb{R}^d, \, t \in (0,\delta]: \quad \mbb{E}^x g(X_t-x) \geq \alpha.
	\end{equation*}
\end{lem}

\begin{proof}
	Denote by $\tau_r^x := \inf\{t>0; X_t \notin B[x,r]\}$ the exit time from $B[x,r]$. Obviously, \begin{align*}
		\mbb{E}^x g(X_t-x)
		&= \mbb{E}^x \big( g(X_t-x) \I_{\{\tau_r^x>t\}} + g(X_t-x) \I_{\{\tau_r^x \leq t\}} \big) \\
		&\geq \inf_{|y-x| \leq r} g(y-x) (1-\mbb{P}^x(\tau_r^x \leq t)) - \|g\|_{\infty} \mbb{P}^x(\tau_r^x \leq t) \\
		&\geq \inf_{|y| \leq r} g(y) - 2 \|g\|_{\infty} \mbb{P}^x(\tau_r^x \leq t).
	\end{align*}
	By (the proof of) the maximal inequality and boundedness of the coefficients of the symbol, we have \begin{equation*}
		\sup_{x \in \mbb{R}^d} \mbb{P}^x(\tau_r^x \leq t) \leq C t \left( 1+ \frac{1}{r^2} \right) 
	\end{equation*}
	for some constant $C>0$ which does not depend on $t,r$. The claim follows by choosing $r>0$ and $\delta>0$ sufficiently small.
\end{proof}

\begin{proof}[Proof of Theorem~\ref{time-1}] \begin{enumerate}
	\item Obviously, it suffices to prove ``$\Rightarrow$''. By Lemma~\ref{time-3}, there exist $\delta>0$, $\alpha \in (0,1)$ such that $\mbb{E}^y g(X_r-y) \geq \alpha$ for all $y \in \mbb{R}^d$ and $r \in (0,\delta]$. Using the Markov property, we get \begin{align*}
		\mbb{E}^x f(X_t)
		&= \mbb{E}^x \left( \mbb{E}^{X_s} f(X_{t-s}) \right) \\
		&= \int_{\Omega} \!\! \int_{\mbb{R}^d} \frac{f((z-y)+y)}{f(y)} f(y) \, \mbb{P}^y_{X_{t-s}}(dz) \big|_{y=X_s} \, d\mbb{P}^x \\
		&\geq \int_{\Omega} \!\! \int_{\mbb{R}^d} f(y) g(z-y) \, \mbb{P}^y_{X_{t-s}}(dz) \big|_{y=X_s} \, d\mbb{P}^x \\
		&\geq \alpha \mbb{E}^xf(X_s)
	\end{align*}
	for all $s \in [t-\delta,t]$. Iterating this procedure gives $\mbb{E}^x f(X_t) \geq \alpha^n \mbb{E}^x f(X_s)$ for any $s \in [t-n \delta,t]$. Choosing $n \in \mbb{N}$ sufficiently large proves $\sup_{s \leq t} \mbb{E}^x f(X_s) \leq \alpha^{-n} \mbb{E}^x f(X_t)$.
	\item We have to check that there exists a suitable function $g$ satisfying \eqref{time-eq4}.  \begin{enumerate}
		\item Since $f(y) \leq c f(y+z) f(-z)$, we have \begin{equation*}
			\inf_{y \in \mbb{R}^d} \frac{f(z+y)}{f(y)} \geq \frac{1}{c} \frac{1}{f(-z)} \geq \min \left\{ 1, \frac{1}{c} \frac{1}{f(-z)} \right\} =: g(z), \qquad z \in \mbb{R}^d.
		\end{equation*}
		Moreover, as $f$ is locally bounded, $\inf_{y \in B[0,r]} g(y)>0$ for $r$ sufficiently small.
		\item $|\log f(z)- \log f(y)| \leq c |z-y|^{\gamma}$ implies \begin{align*}
			\frac{f(z+y)}{f(y)}
			= \exp \left( \log f(z+y)- \log f(y) \right) 
			\geq \exp \left( -c |z|^{\gamma} \right) =: g(z), \qquad z \in \mbb{R}^d. 
		\end{align*}
		\item As $f>c>0$, Hölder continuity of $f$ implies Hölder continuity of $\log f$, and the claim follows from (b).
		\item By the gradient theorem, \begin{align*}
			|f(y+z)-f(y)| 
			= \left| \int_0^1 \nabla f(y+tz) \cdot z \, dt \right|
			\leq |z| \sup_{|z| \leq r} |\nabla f(y+z)|
		\end{align*}
		for all $|z| \leq r$ and $y \in \mbb{R}^d$. Applying the Cauchy--Schwarz inequality gives \begin{equation*}
		\frac{f(z+y)}{f(y)}
			\geq \min \left\{1, 1- |z| \sup_{y \in \mbb{R}^d} \sup_{|z| \leq r} \frac{|\nabla f(y+z)|}{f(y)} \right\} =: g(z).
		\end{equation*}
		\item This is an immediate consequence of \ref{time-1-ii}(d). \qedhere
	\end{enumerate}
\end{enumerate} \end{proof}

The proof of Theorem~\ref{time-1} actually shows that, under the assumptions of Theorem~\ref{time-1}(i), \begin{equation*}
		\sup_{x \in K} \mbb{E}^xf(X_t-x) <\infty \implies \sup_{x \in K} \sup_{s \leq t} \mbb{E}^x f(X_s-x)< \infty
	\end{equation*}
for any set $K \subseteq \mbb{R}^d$.  Next we show that the moments also exist  forward in time provided that $\mbb{E}^x f(X_t-x)$ is bounded in $x$ and $f$ is submultiplicative.

\begin{kor} \label{time-5}
	Let $(X_t)_{t \geq 0}$ be a L\'evy-type process with bounded coefficients and $f: \mbb{R}^d \to (0,\infty)$ a locally bounded measurable submultiplicative function. Then \begin{equation*}
		\exists t>0: \sup_{x \in \mbb{R}^d} \mbb{E}^x f(X_t-x) < \infty \implies \forall s \geq 0:  \sup_{r \leq s} \sup_{x \in \mbb{R}^d} \mbb{E}^xf(X_r-x)<\infty.
	\end{equation*}
\end{kor}

\begin{proof}
	Fix $t>0$ such that $\sup_{x \in \mbb{R}^d} \mbb{E}^x f(X_t-x)<\infty$. It follows from Theorem~\ref{time-1} that $M_1 := 1 \vee \sup_{x \in \mbb{R}^d} \sup_{s \leq t} \mbb{E}^xf(X_s-x)<\infty$. Using the Markov property and the submultiplicativity of $f$, we find \begin{equation*}
		\mbb{E}^x f(X_r-x)
		= \mbb{E}^x \left( \mbb{E}^y f(X_{r-t}-x) \big|_{y=X_t} \right)
		\leq c \mbb{E}^x \left( \mbb{E}^y f(X_{r-t}-y) f(y-x) \big|_{y=X_s} \right)
		\leq c M_1^2
	\end{equation*}
	for all $r \in [t,2t]$ and $x \in \mbb{R}^d$. Hence, $M_2 := 1 \vee \sup_{r \leq 2t} \sup_{x \in \mbb{R}^d} \mbb{E}^x f(X_r-x)<\infty$. By iteration, we obtain $M_{k} := 1 \vee \sup_{r \leq kt} \sup_{x \in \mbb{R}^d} \mbb{E}^x f(X_r-x)<\infty$ for all $k \in \mbb{N}$ and
	\begin{equation*}
		\sup_{x \in \mbb{R}^d} \sup_{r \leq (k+1)t} \mbb{E}^x f(X_r-x) \leq c M_{k}^2< \infty. \qedhere
	\end{equation*}
\end{proof}

\begin{bem}
	If $f$ is not submultiplicative, then Corollary~\ref{time-5} does, in general, not hold true. For a counterexample in the L\'evy case see e.\,g.\ \cite[Remark 25.9]{sato}.
\end{bem}

\section{Existence of moments - sufficient conditions} \label{s-exi}

In this part, we present sufficient conditions for the existence of moments for L\'evy-type processes. Let us recall the corresponding well-known result for L\'evy processes (cf.\ \cite[Theorem 25.3]{sato}): For a L\'evy process $(X_t)_{t \geq 0}$ with L\'evy triplet $(b,Q,\nu)$, we have \begin{equation*}
	\mbb{E}^x f(X_t) < \infty \, \, \text{for some (all) $t>0$}  \iff \int_{|y| \geq 1} f(y) \, \nu(dy)< \infty
\end{equation*}
for any locally bounded measurable submultiplicative function $f: \mbb{R}^d \to (0,\infty)$. In \cite[Theorem 5.11]{ltp} it was observed that for $f(y) := \exp(\zeta y)$, $\zeta \in \mbb{R}^d$, the implication \begin{equation}
	\sup_{x \in \mbb{R}^d} \int_{|y| \geq 1} f(y) \, N(x,dy)<\infty \implies \forall x \in \mbb{R}^d, t \geq 0: \mbb{E}^x f(X_t)<\infty \label{exi-eq1}
\end{equation}
still holds true  for any L\'evy-type process $(X_t)_{t \geq 0}$ with bounded coefficients. In Theorem~\ref{exi-1} we extend this result and show \eqref{exi-eq1} for any function $f \geq 0$ which is comparable to a submultiplicative $C^2$-function. In the second part of this section, we discuss the connection between differentiability of the symbol and existence of moments.


\begin{thm} \label{exi-1} 
	Let $(X_t)_{t \geq 0} \sim (b(x),Q(x), N(x,dy))$ be a L\'evy type process and $K \subseteq \mbb{R}^d$ a compact set. Let $f : \mbb{R}^d \to [0,\infty)$ be a measurable function and $g \in C^2$ submultiplicative such that $g \geq 0$ and $f \asymp g$. Then for any $t>0$ \begin{equation*}
		\sup_{x \in K} \int_{|y| \geq 1} f(y) \, N(x,dy) <\infty \implies \sup_{s \leq t} \sup_{x \in K} \mbb{E}^x f(X_{s \wedge \tau_K}-x)<\infty 
	\end{equation*}
	and  \begin{equation}
		\mbb{E}^x f(X_{t \wedge \tau_K}) \leq C f(x) \exp \left( C (M_1+M_2)t \right) \label{exi-eq2}
	\end{equation}
	where $\tau_K := \inf\{t>0; X_t \notin K\}$ denotes the exit time from the set $K$, $C=C(K)>0$ is a constant (which does not depend on $(X_t)_{t \geq 0}$ and $t$) and
	\begin{equation*}
			M_1 := \sup_{x \in K} \left( |b(x)| + |Q(x)|+ \int_{\mbb{R}^d \backslash \{0\}} (|y|^2 \wedge 1) N(x,dy) \right)<\infty \quad \quad M_2 := \sup_{x \in K} \int_{|y| \geq 1} f(y) \, N(x,dy)<\infty.
		\end{equation*}
		If $(X_t)_{t \geq 0}$ has bounded coefficients, then the claim holds for $K= \mbb{R}^d$. 
\end{thm}

\begin{proof}
	To keep notation simple, we only give the proof for $d=1$. We can assume without loss of generality that $f \in C^2$ is submultiplicative (otherwise replace $f$ by $g$). 
	Let $(\Omega^{\circ}, \mathcal{A}^{\circ},\mathcal{F}_t^{\circ},\mbb{P}^{\circ,x})$, $(W_t^{\circ})_{t \geq 0}$, $(L_t^{\circ})_{t \geq 0}$, $N^{\circ}$ and $k,\sigma$ be as in Theorem~\ref{def-5}. For fixed $R>0$ define an $\mathcal{F}_t^{\circ}$-stopping time by \begin{equation*}
		\tau_R^x := \inf\{t>0; \max\{|X_t^1|,|X_t^2|\} \geq R\}
	\end{equation*}
	and set $\tau := \tau_K \wedge \tau_R^x$. By the submultiplicativity of $f$, we have \begin{equation*}
		f(X_t-X_0) = f(X_t^1+X_t^2) \leq c f(X_t^1) f(X_t^2)
	\end{equation*}
	for some constant $c>0$. Since a submultiplicative function growths at most exponentially, cf.\ \cite[Lemma 25.5]{sato}, there exist constants $a,b>0$ such that \begin{equation*}
		f(X_t-X_0) \leq a \exp \left( b (\sqrt{(X_t^1)^2+1}-1) \right) f(X_t^2) =: h(X_t^1) f(X_t^2).
	\end{equation*}
	Moreover, a straightforward calculation shows  \begin{equation}
		|h'(x)|+|h''(x)| \leq C_1 h(x), \qquad x \in \mbb{R}, \label{exi-eq5}
	\end{equation}
	for some constant $C_1>0$. By It\^o's formula and optional stopping, \begin{align*}
		\mbb{E}^{\circ,x}(&h(X_{t \wedge \tau}^1) f(X_{t \wedge \tau}^2))-a f(0) \\
		&= \mbb{E}^{\circ,x} \left( \int_{[0,t \wedge \tau)} h'(X_{s-}^1) f(X_{s-}^2) b(X_{s-}) \, ds \right)
		+ \frac{1}{2} \mbb{E}^{\circ,x} \left( \int_{[0,t \wedge \tau)} h''(X_{s-}^1) f(X_{s-}^2) \sigma^2(X_{s-}) \, ds \right) \\
		& + \mbb{E}^{\circ,x} \left( \int_{[0,t \wedge \tau)} \int_{|k|\leq 1} f(X_{s-}^2) (h(X_{s-}^1+k(X_{s-},y))-h(X_{s-}^1)-h'(X_{s-}^1) k(X_{s-},y)) \, \nu^{\circ}(dy) \, ds \right) \\
		&+ \mbb{E}^{\circ,x} \left( \int_{[0,t \wedge \tau)}\int_{|k| > 1} h(X_{s-}^1) (f(X_{s-}^2+k(X_{s-},y))-f(X_{s-}^2)) \, \nu^{\circ}(dy) \, ds \right) \\
		&=: I_1+I_2+I_3+I_4.
	\end{align*}
	Recall that $\nu^{\circ}$ denotes the L\'evy measure of the Cauchy process $(L_t^{\circ})_{t \geq 0}$. We estimate the terms separately. By \eqref{exi-eq5} and the definition of $M_1$, it follows easily that \begin{equation*}
		|I_1|+|I_2| \leq C_1 M_1 \mbb{E}^{\circ,x} \left( \int_{[0,t \wedge \tau)} h(X_{s-}^1) f(X_{s-}^2) \, ds \right).
	\end{equation*}
	For $I_4$ we note that by the submultiplicativity of $f$ and \eqref{def-eq20}, \begin{align*}
		|I_4|
		&\leq c \mbb{E}^{\circ,x} \left( \int_{[0,t \wedge \tau)}\!\! \int_{|k|>1} h(X_{s-}^1) f(X_{s-}^2) (1+f(k(X_{s-},y))) \, \nu^{\circ}(dy) \, ds \right) \\
		&\leq c (M_1+M_2) \mbb{E}^{\circ,x} \left( \int_{[0,t \wedge \tau)} h(X_{s-}^1) f(X_{s-}^2) \, ds \right).
	\end{align*}
	It remains to estimate $I_3$. By Taylor's formula, we have \begin{equation*}
		|h(x+z)-h(x)-h'(x)z| \leq \frac{1}{2} |h''(\xi)| z^2
	\end{equation*}
	for some intermediate value $\xi = \xi(x,z) \in (x,x+z)$. Since there exists $C_2>0$ such that $|h''(\xi)| \leq C_2 h(x)$ for all $|z| \leq 1$ and $x \in \mbb{R}$, we get \begin{align*}
		|I_3| 
		\leq C_2 M_1 \mbb{E}^{\circ,x} \left( \int_{[0,t \wedge \tau)} h(X_{s-}^1) f(X_{s-}^2) \, ds \right).
	\end{align*}
	Combining all estimates shows that $\varphi(t) := \mbb{E}^{\circ,x}(h(X_{t \wedge \tau}^1) f(X_{t \wedge \tau}^2) \I_{\{t<\tau\}})$ satisfies \begin{equation*}
		\varphi(t) 
		\leq \mbb{E}^{\circ,x}\left(h(X_{t \wedge \tau}^1) f(X_{t \wedge \tau}^2)  \right) 
		\leq a f(0) + C_3 \int_0^t \varphi(s) \, ds
	\end{equation*}
	for some constant $C_3=C_3(M_1,M_2,f)$. Now it follows from Gronwall's inequality, see e.\,g.\ \cite[Theorem A.43]{bm2}, that $\varphi(t) \leq af(0) e^{C_3 t}$.  Finally, using Fatou's lemma, we can let $R \to \infty$ and obtain \begin{equation*}
		\mbb{E}^{x} f(X_{t \wedge \tau_K}-x) \leq \mbb{E}^{\circ,x}(h(X_{t \wedge \tau_K}^1) f(X_{t \wedge \tau_K}^2)) \leq a f(0) e^{C_3 t}. 
	\end{equation*}
	This proves $\sup_{x \in K} \sup_{s \leq t} \mbb{E}^x f(X_{s \wedge \tau_K}-x)< \infty$; \eqref{exi-eq2} follows from $f(X_t) \leq c f(X_t-x) f(x)$ and the previous inequality.
\end{proof}

\begin{bem}
	The proof of Theorem~\ref{exi-1} shows that the statement holds true for any function $f$ such that there exist $g_1 \in C^2$ submultiplicative, $g_2 \in C^2$ subadditive, $g_1 \geq 0$, $\inf_{x \in \mbb{R}^d} g_2(x)>0$ and $f \asymp g := g_1 \cdot g_2$.
\end{bem}

\begin{bsp} \label{exi-3} \begin{enumerate}
	\item Let $(X_t)_{t \geq 0}$ be a L\'evy-type process with uniformly bounded jumps, i.\,e.\ there exist $R_1,R_2>0$ such that $\spt N(x,\cdot) \subseteq \{y \in \mbb{R}^d; R_1 \leq |y| \leq R_2\}$ for all $x \in \mbb{R}^d$. Then we have \begin{equation*}
		\sup_{x \in \mbb{R}^d} \sup_{s \leq t} \mbb{E}^x f(X_s-x) < \infty \qquad \text{for all $t \geq 0$}
	\end{equation*}
 	for any measurable function $f \geq 0$ which is comparable to a submultiplicative $C^2$-function (e.\,g.\ $f(x) = |x|^{\alpha} \vee 1$, $\alpha>0$, $f(x) = \exp(|x|^{\beta})$, $\beta \in (0,1]$, $f(x) = \log(|x| \vee e)$, \ldots)
	\item Let $(X_t)_{t \geq 0}$ be a stable-like process, that is a L\'evy-type process with symbol $q(x,\xi) = |\xi|^{\alpha(x)}$ for some function $\alpha: \mbb{R}^d \to (0,2)$; for the existence of such processes see \cite{hoh}. If we set $\alpha_l := \inf_{x \in \mbb{R}^d} \alpha(x)$, then, by Theorem~\ref{exi-1}, \begin{equation*}
		\sup_{x \in \mbb{R}^d} \sup_{s \leq t} \mbb{E}^x(|X_s-x|^{\alpha})< \infty \qquad \text{for all $\alpha \in [0,\alpha_l)$.}
	\end{equation*}
\end{enumerate} \end{bsp}

Now we turn to the question whether regularity of the symbol is related to the existence of moments. It is a classical result that for the characteristic function $\chi(\xi) := \mbb{E}e^{i \, \xi X}$ of a random variable $X$, \begin{equation*}
	\text{$\chi$ is $2n$ times differentiable at $\xi =0$} \iff \mbb{E}(X^{2n})<\infty
\end{equation*}
for all $n \in \mbb{N}$. In particular for a L\'evy process $(X_t)_{t \geq 0}$ with characteristic exponent $\psi$ it follows easily from the L\'evy-Khintchine formula that \begin{equation*}
		\text{$\psi$ is $2n$ times differentiable at $\xi =0$} \implies \forall t \geq 0: \mbb{E}(|X_t|^{2n})<\infty.
\end{equation*}

Theorem~\ref{exi-5} below shows that this result can be extended to L\'evy-type processes. For the proof we use the following statement which is of independent interest. To keep notation simple we state the result only in dimension $d=1$; it can be easily extended to higher dimensions by considering $q_j(x,\eta) := q(x,\eta e_j)$, $\eta \in \mbb{R}$, for $j \in \{1,\ldots,d\}$. Here, $e_j$ denotes the $j$-th unit vector in $\mbb{R}^d$.

\begin{lem} \label{exi-6} 
	Let $(q(x,\xi))_{x \in \mbb{R}}$ be a family of negative definite functions with L\'evy-Khintchine representation \eqref{def-eq3} and assume that $q(x,0)=0$ for all $x \in \mbb{R}$. Let $n \in \mbb{N}$ and $K \subseteq \mbb{R}$ be a compact set. Then the following statements are equivalent. \begin{enumerate}
		\item\label{exi-6-i} $q(x,\cdot)$ is $2n$ times differentiable for all $x \in K$, $\xi \in \mbb{R}$ and $\sup_{x \in K} \sup_{\xi \in \mbb{R}} |\partial_{\xi}^{2n} q(x,\xi)| <\infty$.
		\item\label{exi-6-ii} $q(x,\cdot)$ is $2n$ times differentiable at $\xi=0$ for all $x \in K$ and $\sup_{x \in K} |\partial_{\xi}^{2n} q(x,0)| < \infty$.
		\item\label{exi-6-iii} $\sup_{x \in K} \int_{\mbb{R} \backslash \{0\}} y^{2n} \, N(x,dy)<\infty$.
	\end{enumerate}
	In this case, \begin{equation}
		\frac{\partial^k}{\partial \xi^k} q(x,\xi) = \begin{cases} - i b(x) + Q(x) \xi + i \int_{\mbb{R} \backslash \{0\}} (\I_{(0,1]}(|y|))-e^{i \, y \xi}) y \, N(x,dy), & k=1, \\ Q(x) + \int_{\mbb{R} \backslash \{0\}} y^2 e^{i \, y \, \xi} \, N(x,dy), & k=2, \\ i^{k+2} \int_{\mbb{R}\backslash \{0\}} y^k e^{i \, y \xi} \, N(x,dy), & k \in \{3,\ldots,2n\}. \end{cases} \label{exi-eq6}
	\end{equation}
	If $q$ has bounded coefficients, then (i)-(iii) are equivalent for $K = \mbb{R}$.
\end{lem}

\begin{proof}
	Obviously, \ref{exi-6-i} $\Rightarrow$ \ref{exi-6-ii}, so it suffices to prove \ref{exi-6-ii} $\Rightarrow$ \ref{exi-6-iii} $\Rightarrow$ \ref{exi-6-i}. We prove the claim by induction.\par
	
	$n=1$: Suppose that \ref{exi-6-ii} holds true. Using the classical identities \begin{equation}
				\frac{1}{2} = \lim_{y \to 0} \frac{1- \cos(y)}{y^2} \qquad \text{and} \qquad \lim_{h \to 0} \frac{\phi(2h)-2 \phi(0)+\phi(-2h)}{4h^2} = \phi''(0) \label{exi-eq7}
			\end{equation}
			for $\phi$ twice differentiable at $0$, we find by Fatou's lemma \begin{align*}
				\int_{\mbb{R}\backslash \{0\}} y^2 \, N(x,dy)
				&= 2 \int_{\mbb{R}\backslash \{0\}} y^2 \lim_{h \to 0} \frac{1-\cos(2hy)}{(2h y)^2} \, N(x,dy) \\
				&\leq \liminf_{h \to 0} \frac{1}{2h^2} \int_{\mbb{R}} (1-\cos(2h y)) \, N(x,dy) \\
				&= 2 \liminf_{h \to 0} \left( \frac{q(x,2h)+q(x,-2h)}{4h^2} - Q(x) \right) \\
				&= 2 \frac{\partial^2}{\partial \xi^2} q(x,0)- 2 Q(x).
		\end{align*}
		Since $Q$ is locally bounded, cf.\ \cite[Theorem 2.27]{ltp}, \ref{exi-6-iii} follows. On the other hand, if \ref{exi-6-iii} holds, then it is obvious from the L\'evy-Khintchine representation that $q(x,\cdot)$ is twice differentiable and that \eqref{exi-eq6} holds for $k=1,2$.\par
	$n-1 \to n$: Suppose that \ref{exi-6-ii} holds for $n \geq 2$. Then, by the induction hypothesis, we get as in the first part of the proof 
		\begin{align*}
				\int_{\mbb{R}\backslash \{0\}} y^{2n} \, N(x,dy)
				&\leq \liminf_{h \to 0} \frac{1}{2h^2} \int_{\mbb{R}\backslash \{0\}} y^{2(n-1)} (1-\cos(2hy)) \, N(x,dy) \\
				&= 2 (-1)^{n-1} \liminf_{h \to 0} \frac{1}{4h^2} \left( \frac{\partial^{2n-2}}{\partial \xi^{2n-2}} q(x,2h) -2 \frac{\partial^{2n-2}}{\partial \xi^{2n-2}} q(x,0)+ \frac{\partial^{2n-2}}{\partial \xi^{2n-2}} q(x,-2h) \right)  \\
				&= 2 (-1)^{n-1} \frac{\partial^{2n}}{\partial \xi^{2n}} q(x,0).
			\end{align*}
			This shows \ref{exi-6-iii}. If \ref{exi-6-iii} holds, then we can use again the L\'evy-Khintchine representation to conclude that $q(x,\cdot)$ is $2n$ times differentiable, $\sup_{x \in K} \sup_{\xi \in \mbb{R}} |q^{(2n)}(x,\xi)| < \infty$ and that \eqref{exi-eq6} holds.
 \end{proof}
 
Using Lemma~\ref{exi-6}, we obtain the following statement. 

\begin{thm} \label{exi-5}
	Let $(X_t)_{t \geq 0} = (X_t^{(1)},\ldots,X_t^{(d)})_{t \geq 0} \sim (b(x),Q(x),N(x,dy))$ be a L\'evy-type process with symbol $q$ and let $K \subseteq \mbb{R}^d$ be compact. Suppose that $\mathbb{R} \ni \xi \mapsto q_j(x,\xi) := q(x,\xi e_j)$ is $2n$ times differentiable at $\xi = 0$ for all $x \in \mbb{R}^d$ and \begin{equation}
		\left| \frac{\partial^k}{\partial \xi^k} q_j(x,0) \right| \leq c_k (1+|x_j|^k), \qquad k=1,\ldots,2n, \label{exi-eq9}
	\end{equation}
	for some constants $c_k>0$. Then there exist $C_1,C_2>0$ such that \begin{equation*}
		\sup_{x \in K} \sup_{s \leq t} \mbb{E}^x((X_s^{(j)}-x_j)^{2n}) \leq C_1 t e^{C_2 t}  \qquad \text{for all $t \geq 0$.}
	\end{equation*}
\end{thm}

\begin{proof}
	We show the result only for $d=1$; for $d>1$ replace $h$ by $h \cdot e_j$. Throughout this proof, we denote by $L$ the operator \begin{equation*}
		Lf(x) := b(x) f'(x) + \frac{1}{2} Q(x) f''(x) + \int_{\mbb{R}\backslash \{0\}} (f(x+y)-f(x)-f'(x)y \I_{(0,1]}(|y|)) \, N(x,dy), \quad x \in \mbb{R},
	\end{equation*}
	which is well-defined for any $f \in C_b^2(\mbb{R})$. We remind the reader that any function $f \in C_c^2(\mbb{R})$ is contained in the domain of the generator $A$ of $(X_t)_{t \geq 0}$ and that $Af=Lf$. \par
	We prove the claim by induction and start with $n=1$. By Lemma~\ref{exi-6}, $\sup_{x \in K} \int_{\mbb{R}\backslash \{0\}} y^2 \, N(x,dy) < \infty$. Set $f_{h,x}(z) := e^{i \, (z-x)h}-1$ for fixed $h,x \in \mbb{R}$. Using Taylor's formula and the identity \begin{align*}
		Lf_{h,x}(z)+& Lf_{-h,x}(z) \\
		&= -2h \sin((z-x)h) b(z) -2 \cos(h(z-x)) h^2 Q(z) \\
		&+ 2 \int_{\mbb{R}\backslash \{0\}} (\cos((z+y-x)h)-1)-(\cos((z-x)h)-1)+yh \I_{(0,1]}(|y|)) \sin(h(z-x)) \, N(z,dy),
	\end{align*}
	it follows easily that $\sup_{|h| \leq 1} (Lf_{h,x}(z)+Lf_{-h,x}(z))/h^2$ is locally bounded (in $z$). 
	For fixed $R>0$ set $\tau:= \tau_R^x := \inf\{t>0; X_t \notin B(x,R)\}$ and $\varphi(t) := \mbb{E}^x(|X_{t \wedge \tau}-x|^2 \I_{\{t<\tau\}})$. By \eqref{exi-eq7}, \begin{align*}
		\varphi(t)
		\leq \mbb{E}^x(|X_{t \wedge \tau}-x|^2)
		&= 2 \int_{\Omega} |X_{t \wedge \tau}-x|^2 \lim_{h \to 0} \frac{1-\cos(2h  (X_{t \wedge \tau}-x))}{4h^2 (X_{t \wedge \tau}-x)^2} \, d\mathbb{P} \\
		&\leq \liminf_{h \to 0} \frac{1}{4h^2} \left( -\mbb{E}^x e^{i \, 2h  (X_{t \wedge \tau}-x)}+2- \mbb{E}^x e^{-i \, 2h (X_{t \wedge \tau}-x)} \right).
	\end{align*}
	Pick a cut-off function $\chi \in C_c^2(\mbb{R})$ such that $\I_{B(0,1)} \leq \chi \leq \I_{B(0,2)}$. Applying Dynkin's formula to the truncated functions $y \mapsto (-e^{-i 2h (y-x)}+1) \chi(y/n) \in C_c^2(\mbb{R})$ and $y \mapsto (-e^{i2h (y-x)+1}+1) \chi(y/n) \in C_c^2(\mbb{R})$ and letting $n \to \infty$ using the dominated convergence theorem, we find  \begin{equation*}
		\varphi(t)
		\leq \liminf_{h \to 0} \mbb{E}^x \left( \int_{[0,t \wedge \tau)} \frac{Lf_{2h,x}(X_s)+Lf_{-2h,x}(X_s)}{4h^2} \, ds \right).
	\end{equation*}
	By the above considerations, we may apply the dominated convergence theorem and obtain using \eqref{exi-eq7} \begin{equation*}
		\varphi(t)
		\leq  \mbb{E}^x \left( \int_{[0,t \wedge \tau)} \frac{\partial^2}{\partial h^2} Lf_{h,x}(X_s) \bigg|_{h=0} \,ds \right)
		= \mbb{E}^x \left( \int_{[0,t \wedge \tau)} Lg_x(X_s) \, ds \right)
	\end{equation*}
	where $g_x(z) := (z-x)^2$. The growth assumptions \eqref{exi-eq9} for $k=1,2$ imply, by \eqref{exi-eq6}, that \begin{equation*}
		\left| b(z) + \int_{|y| \geq 1} y \, N(z,dy) \right| \leq c_1 (1+|z|) \quad \text{and} \quad Q(z) + \int_{\mbb{R}\backslash \{0\}} y^2 \, N(z,dy) \leq c_2 (1+z^2)
	\end{equation*}
	for all $z \in \mbb{R}$. Therefore it is not difficult to see from the definition of $L$ that there exist constants $C_1,C_2>0$ (which depend (continuously) on $x$, but not on $R$) such that $\varphi$ satisfies the integral inequality \begin{equation*}
		\varphi(t) \leq C_1 t +  C_2 \int_0^t\varphi(s) \, ds \qquad \text{for all $t \geq 0$.}
	\end{equation*}
	By the  Gronwall inequality, cf.\ \cite[Theorem A.43]{bm2}, we get $\varphi(t) \leq C_1 t \exp(C_2 t)$. Since the constants $C_1,C_2$ do not depend on $R$, the claim follows from Fatou's lemma. \par
	Now suppose that $q$ satisfies the assumptions of Theorem~\ref{exi-5} for $n \geq 2$ and that the claim holds true for $n-1$. Then $q(x,\cdot)$ is $2(n-1)$ times differentiable at $\xi=0$ and  it follows from the inductional hypothesis and Lemma~\ref{exi-6} that $\sup_{x \in K} \int_{\mbb{R}\backslash \{0\}} |y|^{2n-2} \, N(x,dy)< \infty$, $\sup_{x \in K} \mbb{E}^x(|X_t-x|^{2n-2})<\infty$ and \begin{equation}
		 \frac{\partial^{2n-2}}{\partial \xi^{2n-2}}q(x,\xi) =Q(x) \delta_{2,n} +  (-1)^{n-1} \int_{\mbb{R}\backslash \{0\}} y^{2n-2} e^{i \, y \xi} \, N(x,dy). \label{exi-eq15}
	\end{equation}
	(Here, $\delta_{k,n}$ denotes the Kronecker delta.) For fixed $h,x \in \mbb{R}$, set $f_{h,x}(z) := (z-x)^{2n-2} (e^{i \, h (z-x)}-1)$. By Taylor's formula and \eqref{exi-eq15}, it is not difficult to see that $\sup_{|h| \leq 1} (Lf_{h,x}(z)+Lf_{-h,x}(z))/h^2$ is locally bounded (in $z$). As in the first part, an application of Fatou's lemma and Dynkin's formula yields \begin{align*}
		\mbb{E}^x(|X_{t \wedge \tau}-x|^{2n})
		&\leq \liminf_{h \to 0} \frac{1}{4h^2} \left( \mbb{E}^x f_{2h,x}(X_{t \wedge \tau}) - f_{2h,x}(0) - f_{-2h,x}(0) + \mbb{E}^x f_{-2h,x}(X_{t \wedge \tau}) \right) \\
		&= \liminf_{h \to 0} \mbb{E}^x \left( \int_{[0,t \wedge \tau)} \frac{Lf_{2h,x}(X_{s})+L f_{-2h,x}(X_s)}{4h^2} \, ds \right).
	\end{align*}
	Since $\sup_{|h| \leq 1} (Lf_{h,x}(z)+Lf_{-h,x}(z))/h^2$ is locally bounded, it follows from the dominated convergence theorem that \begin{align*}
		\mbb{E}^x(|X_{t \wedge \tau}-x|^{2n})
		&\leq \mathbb{E}^x \left( \int_{[0,t \wedge \tau)} \frac{\partial^2}{\partial h^2} Lf_{h,x}(X_s) \big|_{h=0} \, ds  \right) 
		= \mbb{E}^x \left( \int_{[0,t \wedge \tau)} Lg_x(X_s) \, ds \right)
	\end{align*}
	for $g_x(z) := (z-x)^{2n}$. Using again the growth assumptions and Taylor's formula, we find that $\varphi(t) := \mbb{E}^x(|X_{t \wedge \tau}-x|^{2n} \I_{\{t<\tau\}})$ satisfies \begin{equation*}
		\varphi(t) \leq C_1 t + C_2 \int_0^t \varphi(s) \, ds, \qquad t \geq 0,
	\end{equation*}
	for $C_1,C_2>0$ (not depending on $R$). Applying Gronwall's inequality and Fatou's lemma finishes the proof.
\end{proof}

\begin{bem}	 \label{exi-6.5}
	Let $(X_t)_{t \geq 0}$ be a geometric Brownian motion, i.\,e.\ a solution to the SDE \begin{equation*}
			dX_t = \mu X_t \, dt + \sigma X_t \, dB_t
		\end{equation*}
		where $(B_t)_{t \geq 0}$ is a one-dimensional Brownian motion and $\mu \in \mbb{R}$, $\sigma>0$. One can easily verify that 
		\begin{equation*} 
			\mbb{E}^x((X_t-x)^2) = x^2 (e^{2\mu t} (e^{\sigma^2 t}-2)+1). 
		\end{equation*} 
		This means that exponential growth for large $t$ and linear growth for small $t$ is the best we can expect; in this sense the estimate in Theorem~\ref{exi-5} is optimal.
	\end{bem}

\begin{bsp} \label{exi-7}
	Let $(L_t)_{t \geq 0}$ be a ($d$-dimensional) L\'evy process with characteristic exponent $\psi$. Suppose that the L\'evy-driven SDE \begin{equation}
		dX_t = f(X_{t-}) \, dL_t, \qquad X_0 = x, \label{exi-eq17}
	\end{equation}
	has a unique solution $(X_t)_{t \geq 0}$ which is a L\'evy-type process and suppose that its symbol is given by $q(x,\xi) = \psi(f(x)^T \xi)$. If $\psi$ is $2n$-times differentiable at $\xi=0$, i.\,e.\ if $\mbb{E}(|L_t|^{2n})<\infty$, then it follows from Theorem~\ref{exi-5} that $\sup_{s \leq t} \sup_{x \in K} \mbb{E}^x(|X_s-x|^k) <\infty$ for any compact set $K \subseteq \mbb{R}^d$ and $k \leq 2n$. \par
	Important classes of examples are the following: \begin{enumerate}
		\item If $f$ is bounded and locally Lipschitz continuous, then the (unique) solution to \eqref{exi-eq17} is a L\'evy-type process with symbol $q(x,\xi) = \psi(f(x)^T \xi)$, cf.\ \cite{schnurr}. The boundedness of $f$ is needed to ensure that $(X_t)_{t \geq 0}$ is a L\'evy-type process; see \cite[Rem.~3.4]{schnurr} for an example where $f$ is locally Lipschitz continuous, but the solution fails to be a L\'evy-type process.
		\item ($d=2$) The generalized Ornstein-Uhlenbeck process is the solution to the SDE \begin{equation*}
			dX_t = X_{t-} \, dL_t^{(1)} + dL_t^{(2)}, \qquad X_0 = x.
		\end{equation*}
		In \cite[Theorem~3.1]{behme}, it was shown that $(X_t)_{t \geq 0}$ is a L\'evy-type process with symbol $q(x,\xi) = \psi((x,1)^T \xi)$. 
	\end{enumerate}
\end{bsp}

\section{Fractional moments} \label{s-frac}

This section is devoted to estimates of fractional moments, i.\,e.\ we study the small-time and large-time asymptotics of $\mathbb{E}^x \left(\sup_{s \leq t} |X_s-x|^{\alpha} \right)$ for $\alpha>0$. Depending on $\alpha$, there are different techniques to prove such estimates; the following ones have recently been used to obtain estimates for L\'evy processes: \begin{enumerate}
	\item $\alpha \in (0,1]$: bounded variation technique, cf.\ \cite[Theorem 1]{luschgy}.
	\item $\alpha \geq 1$: martingale technique based on the Burkholder--Davis--Gundy inequality, cf.\ \cite[Theorem 1]{luschgy}. 
	\item $\alpha \in (0,2)$: characterization via Blumenthal--Getoor indices, cf.\ \cite[Section 3]{deng}.
\end{enumerate}
Combining the bounded variation and martingale techniques with Theorem~\ref{def-5}, we will extend \cite[Theorem 1]{luschgy} to L\'evy-type processes in the first part of this section (Theorem~\ref{frac-1}, Theorem~\ref{frac-5}). In the second part, we will introduce generalized Blumenthal--Getoor indices and prove extensions of the results presented in \cite{deng}; cf.\ Theorem~\ref{frac-7}, Corollary~\ref{frac-11} and Theorem~\ref{frac-13}. Let us remark that the small-time estimate \begin{equation*}
	\mbb{E}^x \left( \sup_{s \leq t} |X_s-x|^{\alpha} \right) \leq Ct
\end{equation*}
is the best we can expect; otherwise, the Kolmogorov--Chentsov theorem would imply the existence of a modification with exclusively continuous sample paths. \par \medskip

We start with a combination of the bounded variation and martingale technique. A crucial ingredient to obtain estimates  is the Burkholder--Davis--Gundy inequality; for continuous martingales this inequality is standard, but for discontinuous martingales the proof is more involved, see e.\,g.\ \cite{novi85} or \cite{kallen}. The following theorem generalizes \cite[Theorem~3.1]{deng} and the corresponding result in \cite{luschgy}.

\begin{thm} \label{frac-1}
	Let $(X_t)_{t \geq 0} \sim (b(x),Q(x),N(x,dy))$ be a L\'evy-type process with bounded coefficients and suppose that \begin{equation*}
		M := \sup_{x \in \mbb{R}^d}  \left( \int_{|y| >1} |y|^{\alpha} \, N(x,dy)  + \int_{|y| \leq 1} |y|^{\beta} \, N(x,dy) \right)<\infty
	\end{equation*}
	for some $\alpha \in (0,1]$ and $\beta \in [0,2]$. \begin{enumerate}
		\item $\beta \in [1,2]$: Then there exists $C>0$ such that  \begin{align*} 
			\mbb{E}^x \left( \sup_{s \leq t} |X_s-x|^{\kappa} \right) 
			&\leq t^{\kappa} \sup_{x \in \mbb{R}^d} |b(x)|^{\kappa} + C t^{\kappa/2} \sup_{x \in \mbb{R}^d} |Q(x)|^{\kappa/2} \\
			&\quad + C t^{\kappa/\beta} \left(\sup_{x \in \mbb{R}^d} \int_{|y| \leq 1} |y|^{\beta} \, N(x,dy) \right)^{\kappa/\beta} + t^{\kappa/\alpha} \sup_{x \in \mbb{R}^d} \left( \int_{|y| >1} |y|^{\alpha} \, N(x,dy) \right)^{\kappa/\alpha}
		\end{align*}
		for all $t \geq 0$ and $\kappa \in [0,\alpha]$.
		\item $\beta \in [\alpha,1]$:  Then there exists $C>0$ such that  \begin{align*} 
			\mbb{E}^x \left( \sup_{s \leq t} |X_s-x|^{\kappa} \right) 
			&\leq t^{\kappa} \sup_{x \in \mbb{R}^d} \left|b(x)+ \int_{|y| \leq 1} y \, N(x,dy) \right|^{\kappa} + C t^{\kappa/2} \sup_{x \in \mbb{R}^d} |Q(x)|^{\kappa/2} \\
			&\quad + t^{\kappa/\beta} \left(\sup_{x \in \mbb{R}^d} \int_{|y| \leq 1} |y|^{\beta} \, N(x,dy) \right)^{\kappa/\beta} + t^{\kappa/\alpha} \sup_{x \in \mbb{R}^d} \left( \int_{|y| >1} |y|^{\alpha} \, N(x,dy) \right)^{\kappa/\alpha}
		\end{align*}
		for all $t \geq 0$ and $\kappa \in [0,\alpha]$.
		\item $\beta \in [0,\alpha] $: Then there exists $C>0$ such that  \begin{align*} 
			\mbb{E}^x \left( \sup_{s \leq t} |X_s-x|^{\kappa} \right) 
			&\leq t^{\kappa} \sup_{x \in \mbb{R}^d} \left|b(x)+ \int_{|y| \leq 1} y \, N(x,dy) \right|^{\kappa} + C t^{\kappa/2} \sup_{x \in \mbb{R}^d} |Q(x)|^{\kappa/2} \\
			&\quad + t^{\kappa/\alpha} \sup_{x \in \mbb{R}^d} \left( \int_{\mbb{R}^d\backslash \{0\}} |y|^{\alpha} \, N(x,dy) \right)^{\kappa/\alpha}
			\end{align*}
			for all $t \geq 0$ and $\kappa \in [0,\alpha]$.
	\end{enumerate}
\end{thm}

Since any L\'evy-type process with bounded coefficients satisfies $\sup_{x \in \mbb{R}^d} \int_{|y| \leq 1} |y|^2 \, N(x,dy)<\infty$, Theorem~\ref{frac-1}(i) is applicable with $\beta=2$ whenever $\sup_{x \in \mbb{R}^d} \int_{|y|>1} |y|^{\alpha} \, N(x,dy)<\infty$ for some $\alpha \in (0,1]$. Moreover, by the Markov property, Theorem~\ref{frac-1} gives also bounds for $\mathbb{E}^x \left( \sup_{s \leq t} |X_{s+r}-X_r|^{\kappa} \right)$ for any fixed $r \geq 0$. \par

It is well-known that a L\'evy process $(X_t)_{t \geq 0}$ with L\'evy triplet $(0,0,\nu(dy))$ has sample paths of bounded variation and satisfies $\mathbb{E}(|X_t|^{\alpha})<\infty$ if, and only if, $\int_{\mbb{R}^d\backslash \{0\}} |y|^{\alpha} \, \nu(dy)<\infty$ for some $\alpha \in (0,1]$, cf.\ \cite[Theorem 21.9, Theorem 25.3]{sato}. Theorem~\ref{frac-1} extends this statement to L\'evy-type processes. If $(X_t)_{t \geq 0} \sim (0,0,N(x,dy))$ is a L\'evy-type process such that $\sup_{x \in \mbb{R}^d} \int_{\mbb{R}^d\backslash \{0\}} |y|^{\alpha} \, N(x,dy)<\infty$ for some $\alpha \in (0,1]$, then Theorem~\ref{frac-1} shows that $(X_t)_{t \geq 0}$ has $\mbb{P}^x$-almost surely a finite (strong) $p$-variation on compact $t$-intervals for any $p>\alpha$ and $x \in \mbb{R}^d$. 

\begin{proof}[Proof of Theorem~\ref{frac-1}]
	Because of Jensen's inequality, it suffices to prove the claim for $\kappa = \alpha$. By Theorem~\ref{def-5}, there exist a Markov extension $(\Omega^{\circ},\mathcal{A}^{\circ},\mc{F}_t^{\circ},\mbb{P}^{\circ,x})$, a Brownian motion $(W_t^{\circ})_{t \geq 0}$, a Cauchy process $(L_t^{\circ})_{t \geq 0}$ with jump measure $N^{\circ}$ and $k,\sigma$ such that \eqref{def-eq20} holds and 
	\begin{equation} \label{frac-eq3} \begin{aligned}
		X_t-x &= \int_0^t b(X_{s-}) \, ds + \int_0^t \sigma(X_{s-}) \, dW_s^{\circ}+ \int_0^t \!\! \int_{|k|>1} k(X_{s-},z) \, N^{\circ}(dz,ds) \\ 
		& \quad + \int_0^t \!\! \int_{|k| \leq 1} k(X_{s-},z) \, (N^{\circ}(dz,ds)-\nu^{\circ}(dz) \, ds) 
	\end{aligned} \end{equation}
	$\mbb{P}^{\circ,x}$-almost surely. 
	First, we prove (i). By \eqref{frac-eq3}, we have \begin{align*}
		X_t-x 
		&= \int_0^t b(X_{s-}) \, ds + \int_0^t \sigma(X_{s-}) \, dW_s^{\circ}+ \sum_{0 \leq s \leq t} k(X_{s-},\Delta L_s^{\circ}) \I_{\{|k(X_{s-},\Delta L_s^{\circ})|>1\}} \\
		&\quad + \int_0^t \!\! \int_{|k| \leq 1} k(X_{s-},z) \, (N^{\circ}(dz,ds)-\nu^{\circ}(dz,ds)).
	\end{align*}
	Using the elementary estimate $(u+v)^{\alpha} \leq u^{\alpha}+v^{\alpha}$, $u,v \geq 0$, $\alpha \in (0,1]$, yields \begin{align*}
		\sup_{s \leq t} |X_s-x|^{\alpha} 
		&\leq \sup_{s \leq t} \left| \int_0^s b(X_{r-}) \, dr \right|^{\alpha} +\sup_{s \leq t} \left| \int_0^s \sigma(X_{r-}) \, W_r^{\circ} \right|^{\alpha} + \sum_{0 \leq s \leq t} |k(X_{s-},\Delta L_s^{\circ})|^{\alpha} \I_{\{|k(X_{s-},\Delta L_s^{\circ})|>1\}} \\ 
		&\quad + \sup_{s \leq t} \left| \int_0^s \!\! \int_{|k| \leq 1} k(X_{r-},z) \, (N^{\circ}(dz,dr)-\nu^{\circ}(dz,dr)) \right|^{\alpha} \\
		&\leq \sup_{x \in \mbb{R}^d} |b(x)|^{\alpha} t^{\alpha} + \sup_{s \leq t} \left| \int_0^s \sigma(X_{r-}) \, dW_r^{\circ} \right|^{\alpha}+ \int_0^t \!\! \int_{|k|>1} |k(X_{s-},z)|^{\alpha} \, N^{\circ}(dz,ds) \\ 
		&\quad + \sup_{s \leq t} \left| \int_0^s \!\! \int_{|k| \leq 1} k(X_{r-},z) \, (N^{\circ}(dz,dr)-\nu^{\circ}(dz,dr)) \right|^{\alpha} .
	\end{align*}
	Integrating both sides and using that, by Jensen's inequality, \begin{equation*}
		\mathbb{E}^{\circ,x} \left( \sup_{s \leq t} \left| \int_0^s \sigma(X_{r-}) \, W_r^{\circ} \right|^{\alpha} \right)
		\leq \mathbb{E}^{\circ,x} \left( \sup_{s \leq t} \left| \int_0^s \sigma(X_{r-}) \, W_r^{\circ} \right|^{2} \right)^{\alpha/2}
	\end{equation*}
	and \begin{align*}
		\mathbb{E}^{\circ,x} \bigg( \sup_{s \leq t}& \left| \int_0^s \!\! \int_{|k| \leq 1} k(X_{r-},z) \, (N^{\circ}(dz,dr)-\nu^{\circ}(dz,dr)) \right|^{\alpha} \bigg) \\
		&\leq \mathbb{E}^{\circ,x} \bigg( \sup_{s \leq t} \left| \int_0^s \!\! \int_{|k| \leq 1} k(X_{r-},z) \, (N^{\circ}(dz,dr)-\nu^{\circ}(dz,dr)) \right|^{\beta} \bigg)^{\alpha/\beta}, 
	\end{align*}
	the assertion follows (for $\kappa = \alpha$) from It\^o's isometry and the Burkholder--Davis--Gundy inequality \cite[Theorem~1]{novi85}. \par 
	If $\beta \in [0,1]$, then \begin{equation*}
		\mathbb{E}^{\circ,x} \left( \int_0^t  \!\! \int_{|k| \leq 1} |k(X_{s-},z)| \, \nu^{\circ}(dz) \, ds \right) 
		= \mathbb{E}^{\circ,x} \left( \int_0^t \!\! \int_{|y| \leq 1} |y|  \, N(X_{s-},dy) \, ds \right) 
		\leq Mt < \infty.
	\end{equation*}
	Therefore, we can write
	\begin{align*}
		X_t-x 
		&=  \int_0^t \tilde{b}(X_{s-}) \, ds + \int_0^t \sigma(X_{s-}) \, dW_s^{\circ} + \int_0^t \!\! \int_{\mbb{R}^d\backslash \{0\}} k(X_{s-},z) \, N^{\circ}(dz,ds) \notag \\
		&= \int_0^t \tilde{b}(X_{s-}) \, ds + \int_0^t \sigma(X_{s-}) \, dW_s^{\circ}+  \sum_{0 \leq s \leq t} k(X_{s-},\Delta L_s^{\circ}) 
	\end{align*}
	where $\tilde{b}(x) := b(x) + \int_{|y| \leq 1} y \, N(x,dy)$. The bounds for drift and diffusion are obtained as in the first part of this proof; it remains to estimate the jump part. If $1 \geq \beta \geq \alpha$, then another application of the inequality $(u+v)^{\beta} \leq u^{\beta}+v^{\beta}$ and Jensen's inequality yield \begin{align*}
		\mbb{E}^{\circ,x} \left( \sup_{s \leq t}\left| \sum_{0 \leq r \leq s} k(X_{r-},\Delta L_r) \I_{\{|k(X_{r-},\Delta L_r)| \leq 1\}} \right|^{\alpha} \right)
		&\leq \mbb{E}^{\circ,x} \left(\sup_{s \leq t} \left|\sum_{0 \leq r \leq s} k(X_{r-},\Delta L_r) \I_{\{|k(X_{r-},\Delta L_r)| \leq 1\}} \right|^{\beta} \right)^{\alpha/\beta} \\
		&\leq \mbb{E}^{\circ,x} \left( \sum_{0 \leq r \leq t} |k(X_{r-},\Delta L_r)|^{\beta} \I_{\{|k(X_{r-},\Delta L_r)| \leq 1\}} \right)^{\alpha/\beta} \\
		&= \mbb{E}^{\circ,x} \left( \int_0^t \!\! \int_{|y| \leq 1} |y|^{\beta} \, N(X_{s-},dy) \, ds \right)^{\alpha/\beta}.
	\end{align*}
	Estimating the large jumps in exactly the same way (but without applying Jensen's inequality), we get \begin{equation*}
		\mbb{E}^{\circ,x} \left( \sup_{s \leq t}\left| \sum_{0 \leq r \leq s} k(X_{r-},\Delta L_r) \right|^{\alpha} \right) \leq t^{\alpha/\beta} \sup_{x \in \mbb{R}^d} \left[ \int_{|y| \leq 1} |y|^{\beta} \, N(x,dy) \right]^{\alpha/\beta} + t \sup_{x \in \mbb{R}^d} \int_{|y| >1} |y|^{\alpha} \, N(x,dy).
	\end{equation*}
	Finally, if $\beta \in [0,\alpha]$, then $M<\infty$ for $\beta = \alpha$, and the claim follows from (ii).
\end{proof}	

\begin{bem} \label{frac-3} \begin{enumerate} 
	\item Let $\kappa \in (0,1]$ be such that $\inf_{\theta \in [\kappa,1]} \sup_{x \in \mbb{R}^d} \int_{\mbb{R}^d\backslash \{0\}} |y|^{\theta} \, N(x,dy)<\infty$ and \begin{equation*}
		q(x,\xi) = \int_{\mbb{R}^d \backslash \{0\}} (1-e^{iy \xi}) \, N(x,dy).
	\end{equation*}
	Then an application of Jensen's inequality and Theorem~\ref{frac-1} show that \begin{equation*}
		\mbb{E}^x \left( \sup_{s \leq t} |X_s-x|^{\kappa} \right) \leq \inf_{\theta \in [\kappa,1]} \left( t \sup_{x \in \mbb{R}^d} \int_{\mbb{R}^d \backslash \{0\}} |y|^{\theta} \, N(x,dy) \right)^{\kappa/\theta}
	\end{equation*}
	for any L\'evy-type process $(X_t)_{t \geq 0}$ with symbol $q$. This generalizes \cite[Theorem~3.2]{deng} where the inequality was proved for L\'evy processes.
	\item Let $(X_t)_{t \geq 0}$ be a L\'evy-type process and $\alpha \in (0,1]$ such that \begin{equation}
		f(t) := \sup_{s \leq t} \sup_{x \in \mbb{R}^d} \mbb{E}^x(|X_s-x|^{\alpha})<\infty \qquad \text{for all $t \geq 0$.} \tag{$\star$} \label{frac-st1}
	\end{equation}
	Then $\lim_{t \to \infty} f(t)/t$ exists and is finite. \par \medskip
	
	\emph{Indeed:} Since $(u+v)^{\alpha} \leq u^{\alpha}+v^{\alpha}$, $u,v \geq 0$, the Markov property gives \begin{equation*}
			\mbb{E}^x(|X_{t+s}-x|^{\alpha}) 
			\leq \mbb{E}^x  \left[ \mbb{E}^{X_t}(|X_s-X_0|^{\alpha}) \right] + \mbb{E}^x(|X_t-x|^{\alpha})
			\leq f(s)+f(t).
		\end{equation*}
		Consequently, $f$ is subadditive. Applying \cite[Theorem~6.6.4]{hille} finishes the proof. \par \medskip
		
		Note that, by Theorem~\ref{frac-1}, assumption \eqref{frac-st1} is, in particular, satisfied if $(X_t)_{t \geq 0}$ has bounded coefficients and $\alpha,\beta$ are as in Theorem~\ref{frac-1}.
	\end{enumerate}
\end{bem} 

The Burkholder--Davis--Gundy inequality yields also estimates of fractional moments for $\alpha \geq 1$: 

\begin{thm} \label{frac-5}
	Let $(X_t)_{t \geq 0} \sim (b(x),Q(x),N(x,dy))$ be a L\'evy-type process with bounded coefficients and $\alpha \geq 1$, $\beta \in [1,2]$ such that \begin{equation*}
		M := \sup_{x \in \mbb{R}^d} \left( \int_{|y| \leq 1} |y|^{\beta} \, N(x,dy) + \int_{|y| >1} |y|^{\alpha} \, N(x,dy) \right) <\infty.
	\end{equation*}
	\begin{enumerate}
		\item If $\alpha \in [1,2]$, then there exists $C>0$ such that \begin{align*}
			\mbb{E}^x \left( \sup_{s \leq t}  |X_s-x|^{\kappa} \right) 
			&\leq C  \sup_{x \in \mbb{R}^d} \bigg(t^{\kappa} \left|b(x)+ \int_{|y| >1} y \, N(x,dy) \right|^{\kappa} +  t^{\kappa/2} |Q(x)|^{\kappa/2} \bigg) \\
			&\quad + C \sup_{x \in \mbb{R}^d} \bigg( t^{\kappa/\beta} \left[  \int_{|y| \leq 1} |y|^{\beta} \, N(x,dy) \right]^{\kappa/\beta}
			+  t^{\kappa/\alpha} \left[\int_{|y| >1} |y|^{\alpha} \, N(x,dy) \right]^{\kappa/\alpha} \bigg)
		\end{align*}
		for all $t \geq 0$ and $\kappa \in [0,\alpha \wedge \beta]$. 
		\item If $\alpha>2$, then there exists $C>0$ such that \begin{align*}
			\mbb{E}^x \left( \sup_{s \leq t}  |X_s-x|^{\kappa} \right) 
			&\leq C \sup_{x \in \mbb{R}^d} \left( t^{\kappa} \left|b(x)+ \int_{|y| >1} y \, N(x,dy) \right|^{\kappa} + t^{\kappa/2}|Q(x)|^{\kappa/2} \right) \\
			&\quad + C \sup_{x \in \mbb{R}^d} \left( t^{\kappa/\alpha} \left[\int_{\mbb{R}^d \backslash \{0\}} |y|^{\alpha} \, N(x,dy) \right]^{\kappa/\alpha} + t^{\kappa/2} \left[ \int_{\mbb{R}^d \backslash \{0\}} |y|^2 \, N(x,dy) \right]^{\kappa/2} \right)
		\end{align*}
		for all $t \geq 0$ and $\kappa \in [0,\alpha]$. 
		\item (Wald's identity) Suppose that $q$ is of martingale-type, i.\,e. \begin{equation*}
			q(x,\xi) = \int_{\mbb{R}^d \backslash \{0\}} (1-e^{iy \cdot \xi}+iy \cdot \xi) \, N(x,dy),
		\end{equation*}
		and $\sup_{x \in \mbb{R}^d} \int_{\mbb{R}^d \backslash \{0\}} |y|^{\alpha} \, N(x,dy)<\infty$ for some $\alpha \in [1,2]$. Then $\mbb{E}^x(X_{\tau})=x$ holds for any stopping time $\tau$ such that $\mbb{E}^x(\tau^{1/\alpha})<\infty$.
	\end{enumerate}
\end{thm}

\begin{proof}
	As in the proof of Theorem~\ref{frac-1}, we fix a Markov extension $(\Omega^{\circ},\mathcal{A}^{\circ},\mc{F}_t^{\circ},\mbb{P}^{\circ,x})$, a Brownian motion $(W_t^{\circ})_{t \geq 0}$, a Cauchy process $(L_t^{\circ})_{t \geq 0}$ with jump measure $N^{\circ}$ and $k,\sigma$ such that \eqref{def-eq20} and \eqref{frac-eq3} hold. As \begin{equation*}
		\mbb{E}^{\circ,x} \left( \int_0^t \!\! \int_{|k| > 1} |k(X_{s-},z)| \, \nu^{\circ}(dz) \, ds \right) 
		= \mbb{E}^{\circ,x} \left( \int_0^t \!\! \int_{|y| > 1} |y| \, N(x,dy) \right)
		\leq M t,
	\end{equation*}
	we can write \begin{align*}
		X_t-x 
		&= \int_0^t \bar{b}(X_{s-}) \,ds + \int_0^t \sigma(X_{s-}) \, dW_s^{\circ}+ \int_0^t \!\! \int_{|k|>1} k(X_{s-},z) \, (N^{\circ}(dz,ds)-\nu^{\circ}(dz) \, ds) \\
		&\quad + \int_0^t \!\! \int_{|k| \leq 1} k(X_{s-},z) \, (N^{\circ}(dz,ds)-\nu^{\circ}(dz) \, ds) 
	\end{align*} 
	where $\bar{b}(x) := b(x) + \int_{|y| >1} y \, N(x,dy)$. Note that $\bar{b}$ is well-defined since \begin{equation*}
		\int_{|y| >1} |y| \, N(x,dy) \stackrel{\alpha \geq 1}{\leq} \int_{|y|>1} |y|^{\alpha} \, N(x,dy) \leq M < \infty.
	\end{equation*}
	Using the elementary estimate \begin{equation*}
		\left( \sum_{i=1}^4 u_i \right)^{\alpha} \leq 4^{\alpha-1} \sum_{i=1}^4 u_i^{\alpha}, \qquad u_i \geq 0,
	\end{equation*}
	(i) and (ii) follow from \cite[Theorem~1]{novi85}, \eqref{def-eq20}, Jensen's inequality and It\^o's isometry for $\kappa = \alpha$. Again we apply Jensen's inequality to obtain the estimate for $\kappa \in [0,\alpha]$. Wald's identity is a direct consequence of \cite[Theorem~2]{novi85}.
\end{proof}

\begin{bem} \label{frac-6}
	By Theorem~\ref{frac-1} and Theorem~\ref{frac-5}, \begin{equation*}
		\sup_{x \in \mbb{R}^d} \int_{\mbb{R}^d \backslash \{0\}} |y|^{\alpha} \, N(x,dy) < \infty \implies \forall t \leq 1: \sup_{x \in \mbb{R}^d} \mbb{E}^x \left( \sup_{s \leq t} |X_s-x|^{\alpha} \right) \leq Ct
	\end{equation*}
	for any L\'evy-type process $(X_t)_{t \geq 0} \sim (0,0,N(x,dy))$ and $\alpha >0$. One can show that at least a partial converse holds true: \begin{equation*}
		\forall t \leq 1: \mbb{E}^x(|X_t-x|^{\alpha}) \leq Ct \implies \int_{\mbb{R}^d \backslash \{0\}} |y|^{\alpha} \, N(x,dy)<\infty.
	\end{equation*}
	This follows by combining the integrated heat kernel estimate \eqref{intro-eq1} with Fatou's lemma and the identity \begin{equation*}
		\int_{|y| \geq 1} |y|^{\alpha} \, N(x,dy) = \alpha \int_{[1,\infty)} r^{\alpha-1} N(x,\{y \in \mbb{R}^d; |y| \geq r\}) \, dr,
	\end{equation*}
	see \cite[Proposition 3.10]{ihke} for details.
\end{bem}

In the last part of this section, we describe the asymptotics of fractional moments in terms of the growth of the symbol. To this end, we recall the notion of Blumenthal--Getoor indices. Blumenthal and Getoor \cite{blumen} introduced various indices for L\'evy processes; we will use the following ones: For a L\'evy process $(L_t)_{t \geq 0}$ with characteristic exponent $\psi$ and L\'evy triplet $(b,Q,\nu)$, we call 
\begin{equation} \label{frac-eq5} \begin{aligned}
	\beta_0 
	&:= \sup \left\{\alpha \geq 0; \lim_{|\xi| \to 0} \frac{|\psi(\xi)|}{|\xi|^{\alpha}} = 0 \right\}
	= \sup \left\{\alpha \geq 0; \limsup_{|\xi| \to 0} \frac{|\psi(\xi)|}{|\xi|^{\alpha}} < \infty \right\},  \\
	\beta_{\infty} 
	&:= \inf \left\{\alpha \geq 0; \lim_{|\xi| \to \infty} \frac{|\psi(\xi)|}{|\xi|^\alpha} =0\right\}
	= \inf \left\{\alpha \geq 0; \limsup_{|\xi| \to \infty} \frac{|\psi(\xi)|}{|\xi|^\alpha} < \infty \right\}
\end{aligned} \end{equation}
the \emph{Blumenthal--Getoor index at $0$ and $\infty$}, respectively. Then $\beta_0, \beta_{\infty} \in [0,2]$ and \begin{equation*}
	\beta_0 
	= \sup \left\{ \alpha \leq 2; \int_{|y| \geq 1} |y|^{\alpha} \,\nu(dy) < \infty \right\}
	= \sup \left\{ \alpha \leq 2; \mbb{E}|L_t|^{\alpha} < \infty \right\}.
\end{equation*}
For a proof of the first equality see e.\,g.\ \cite[Proposition~5.4]{rs97}; the second equality follows from \cite[Theorem~25.3]{sato}. There are several ways to define so-called generalized Blumenthal--Getoor indices for L\'evy-type processes, cf.\ \cite{rs97} and \cite[Section 5.2]{ltp}. Following \cite{ltp}, we define for a family $(q(x,\xi))_{x \in \mbb{R}^d}$ of characteristic exponents the \emph{generalized Blumenthal--Getoor index at $0$ and $\infty$}, respectively, as 
\begin{equation} \label{frac-eq6} \begin{aligned}
	\beta_0^x &:= \sup \left\{\alpha \geq 0; \limsup_{|\xi| \to 0} \frac{1}{|\xi|^{\alpha}} \sup_{|y-x| \leq |\xi|^{-1}} \sup_{|\eta| \leq |\xi|} |q(y,\eta)| < \infty \right\}, \\
	\beta_{\infty}^x &:= \inf \left\{\alpha \geq 0;\limsup_{|\xi| \to \infty} \frac{1}{|\xi|^{\alpha}} \sup_{|y-x| \leq |\xi|^{-1}} \sup_{|\eta| \leq |\xi|} |q(y,\eta)| < \infty \right\}
\end{aligned} \end{equation}
for $x \in \mbb{R}^d$. The next theorem is one of our main results.

\begin{thm} \label{frac-7}
	Let $(X_t)_{t \geq 0}$ be a L\'evy-type process with symbol $q$ and let $x \in \mbb{R}^d$.  Suppose that there exist $\alpha,\beta \in (0,2]$, $\gamma<\beta$ and $C>0$ such that \begin{align*}
		|q(y,\xi)| \leq C (1+|y|^{\gamma}) |\xi|^{\beta}, \quad \quad \text{for all} \, \, \, |\xi| \leq 1, \, |y-x| \leq |\xi|^{-1},   \\
		|q(y,\xi)| \leq C(1+|y|^{\gamma}) |\xi|^{\alpha}, \quad\quad \text{for all} \, \, \, |\xi| \geq 1, \, |y-x| \leq |\xi|^{-1}.
	\end{align*}
	Then \begin{equation*}
		\mbb{E}^x \left( \sup_{s \leq t} |X_s-x|^{\kappa} \right) \leq C f(t)^{\kappa/\gamma}  \qquad \text{for all $t \leq 1$, $\kappa \in [0,\gamma]$,}
	\end{equation*}
	where $C=C(x,\gamma,\alpha,\beta)$ and $f(t) := t^{\frac{\gamma}{\alpha} \wedge 1}$.
\end{thm}

Note that under the assumptions of Theorem~\ref{frac-7} we can choose for any $\kappa \in (0,\beta)$ some $\gamma<\beta$ such that $\kappa \in (0,\gamma]$; therefore Theorem~\ref{frac-7} gives moment estimates for all $\kappa \in (0,\beta)$. 

\begin{proof}[Proof of Theorem~\ref{frac-7}]
	Throughout this proof the constant $C_1=C_1(\gamma,\alpha,\beta)>0$ may vary from line to line. Without loss of generality, we may assume that $\gamma \neq \alpha$ and $\kappa \in [0,\gamma]$ (otherwise we choose $\gamma <\beta$ sufficiently large such that these two relations are satisfied). Again, we denote by $\tau^x(r) := \tau_r^x $ the exit time from $B(x,r)$. Fix $R>0$. By Lemma~\ref{frac-9}, \begin{align*}
		\mbb{E}^x \left( \sup_{s \leq t \wedge \tau_R^x} |X_s-x|^{\gamma} \right)
		&= \int_0^{\infty} \mbb{P}^x \left( \sup_{s \leq t \wedge \tau_R^x} |X_s-x| \geq r^{1/\gamma} \right) \, dr \\
		&\leq  \int_0^{\infty} \min \left\{1, \, C_1 \mbb{E}^x \left(\int_{[0,t \wedge \tau^x(r^{1/\gamma}) \wedge \tau_R^x)}\sup_{|\xi| \leq r^{-1/\gamma}} |q(X_{s},\xi)| \, ds \right) \right\} \, dr.
	\end{align*}
	Using the growth assumptions on $q$, we get \begin{align*}
		\mbb{E}^x \left( \sup_{s \leq t} |X_s-x|^{\gamma} \right)
		&\leq  \int_0^{t^{\gamma/\alpha}}1 \, dr + C_1 \int_{t^{\gamma/\alpha}}^{\infty}   \mbb{E}^x \left(\int_{[0,t \wedge \tau^x(r^{1/\gamma}) \wedge \tau_R^x)} \sup_{|\xi| \leq r^{-1/\gamma}} |q(X_s,\xi)| \, ds \right) \, dr  \\
		&\leq  t^{\gamma/\alpha} +C_1 \left(\int_{t^{\gamma/\alpha}}^1 r^{-\alpha/\gamma} \, dr+\int_{1}^{\infty} r^{-\beta/\gamma} \, dr \right) \mbb{E}^x\left( \int_{[0,t \wedge \tau_R^x)} (1+|X_s|^{\gamma}) \, ds \right) \\
		&\leq  f(t) + C_1 \int_0^t (1+\mbb{E}^x(|X_{s}|^{\gamma} \I_{\{s<\tau_R^x\}})) \, ds
	\end{align*} 
	for all $t \leq 1$. This shows that $\varphi(t) := \mbb{E}^x \left( \sup_{s \leq t \wedge \tau_R^x} |X_{s}-x|^{\gamma} \I_{\{t<\tau_R^x\}} \right)$ satisfies \begin{equation*}
		\varphi(t) 
		\leq \mbb{E}^x \left( \sup_{s \leq t \wedge \tau_R^x} |X_s-x|^{\gamma} \right)
		\leq C_1  f(t) (1+|x|^{\gamma})+ C_1 \int_0^t \varphi(s) \, ds.
	\end{equation*}
	Hence, by Gronwall's inequality, \begin{align*}
		\varphi(t) \leq C_1  f(t) (1+|x|^{\gamma}) \exp \left(C_1 t \right).
	\end{align*}
	Finally, since the constant $C_1$ does not depend on $R$, we can let $R \to \infty$ using Fatou's lemma. For $\kappa \in [0,\gamma]$ apply Jensen's inequality.
\end{proof}

\begin{bsp} \label{frac-10}
	Let $(X_t)_{t \geq 0}$ be a L\'evy-type process which is a solution to an SDE of the form \begin{equation*}
		dX_t = f(X_{t-}) \, dL_t, \qquad X_0 = x,
	\end{equation*}
	where $(L_t)_{t \geq 0}$ is a L\'evy process with characteristic exponent $\psi$ and $f$ a function of sublinear growth, i.\,e.\ $|f(x)| \leq C (1+|x|^{1-\epsilon})$ for some $C,\epsilon>0$. Denote by $\beta_0$ and $\beta_{\infty}$ the Blumenthal--Getoor indices of $\psi$ at $0$ and $\infty$, cf.\ \eqref{frac-eq5}. Then \begin{equation*}
		\mbb{E}^x \left( \sup_{s \leq t} |X_s-x|^{\kappa} \right) \leq  C' t^{\kappa/\beta_{\infty} \wedge 1} \qquad \text{for all $t \leq 1$, $\kappa \in [0,\beta_0)$, $\kappa \neq \beta_{\infty}$}.
	\end{equation*}
\end{bsp}
	
\begin{kor} \label{frac-11} 
	Let $(X_t)_{t \geq 0}$ be a L\'evy-type process with symbol $q$. Assume that \begin{equation}
		\limsup_{|\xi| \to \infty} \frac{1}{|\xi|^{\alpha}} \sup_{|y-x| \leq |\xi|^{-1}} \sup_{|\eta| \leq |\xi|} |q(y,\eta)| < \infty  \label{frac-eq8}
	\end{equation}
	for some $\alpha \in (0,2]$. Then \begin{equation*}
		\mbb{E}^x \left( \sup_{s \leq t} |X_s-x|^{\kappa} \right) \leq \begin{cases} C t^{\frac{\kappa}{\alpha} \wedge 1}, & \kappa \neq \alpha, \\ Ct |\log t|, & \kappa = \alpha \end{cases}
	\end{equation*}
	for all $t \leq 1$ and $\kappa \in [0,\beta_0^x)$. Here $\beta_0^x$ and $\beta_{\infty}^x$ denote the generalized Blumenthal--Getoor indices at $0$ and $\infty$, respectively, cf. \eqref{frac-eq6}.
\end{kor}

\begin{proof}
	Because of the assumptions on $q$, the growth conditions in  Theorem~\ref{frac-7} are satisfied for any $\gamma \in (0,\beta)$; in particular we can choose $\gamma = \kappa$. 
\end{proof}

\begin{bem} \begin{enumerate}
	\item Applying Corollary~\ref{frac-11} to $\alpha$-stable and tempered $\alpha$-stable processes shows that the estimates are optimal, cf.\ \cite[p.~431-32]{luschgy}.  
	\item By the very definition of the Blumenthal--Getoor index (see \eqref{frac-eq6}), we know that the limit \begin{equation*}
		\limsup_{|\xi| \to \infty} \frac{1}{|\xi|^{\alpha}} \sup_{|y-x| \leq |\xi|^{-1}} \sup_{|\eta| \leq |\xi|} |q(y,\eta)|
	\end{equation*}
	is finite (infinite) if $\alpha>\beta_{\infty}^x$ (if $\alpha<\beta_{\infty}^x$). Therefore, \eqref{frac-eq8} is violated for any $\alpha \in (0,\beta_{\infty}^x)$ and automatically satisfied for $\alpha \in (\beta_{\infty}^x,2]$. The case $\alpha = \beta_{\infty}^x$ has to be checked individually.
	\item 
	Combining Corollary~\ref{frac-11} and Fatou's lemma shows that \begin{equation*}
		\liminf_{t \to 0} \frac{1}{t^{1/\alpha}} \sup_{s \leq t} |X_s-x| = 0 \quad \text{$\mbb{P}^x$-a.s.}
	\end{equation*}
	for any $\alpha>\beta_{\infty}^x$; see also \cite[Theorem 5.16]{ltp}.
	\item Corollary~\ref{frac-11} can be proved using a very similar argument as in the proof of Theorem~\ref{frac-13}. The proof then shows in particular that the estimate\begin{equation*}
		\mbb{E}^x \left( \sup_{s \leq t} |X_s-x|^{\kappa} \wedge 1 \right) \leq \begin{cases} C t^{\frac{\kappa}{\alpha} \wedge 1}, & \kappa \neq \alpha, \\ Ct |\log t|, & \kappa = \alpha \end{cases}
	\end{equation*}
	holds true for any $\kappa>0$ and $t \leq 1$.
\end{enumerate} \end{bem}

There is an analogous result for the large-time asymptotics; it extends \cite[Theorem 3.3]{deng}.

\begin{thm} \label{frac-13}
	Let $(X_t)_{t \geq 0}$ be a L\'evy-type process with symbol $q$ and $\beta \in (0,2]$ such that \begin{align*}
		\limsup_{|\xi| \to 0} \frac{1}{|\xi|^{\beta}} \sup_{|x-y| \leq |\xi|^{-1}} \sup_{|\eta| \leq |\xi|} |q(x,\eta)| &<\infty,
	\end{align*}
	then \begin{equation*}
		\mbb{E}^x \left( \sup_{s \leq t} |X_s-x|^{\kappa} \right) \leq C t^{\kappa/\beta} \qquad \text{for all $t \geq 1$, $\kappa \in [0,\beta)$}.
	\end{equation*}
\end{thm}

\begin{proof}
	An application of the maximal inequality \eqref{frac-eq11} yields \begin{align*}
		\mbb{E}^x \left( \sup_{s \leq t} |X_s-x|^{\kappa} \right) 
		&= \int_0^{\infty} \mbb{P}^x \left( \sup_{s \leq t} |X_s-x| \geq r^{1/\kappa} \right) \, dr \\
		&\leq \int_0^{\infty} \min \left\{1, C t \sup_{|y-x| \leq r^{1/\kappa}} \sup_{|\eta| \leq r^{-1/\kappa}} |q(y,\eta)| \right\} \, dr.
	\end{align*}
	Hence, \begin{equation*}
		\mbb{E}^x \left( \sup_{s \leq t} |X_s-x|^{\kappa} \right) 
		\leq \int_0^{t^{\kappa/\beta}} 1 \, dr + C' t \int_{t^{\kappa/\beta}}^{\infty} r^{-\beta/\kappa} \, dr
		= O(t^{\kappa/\beta}). \qedhere
	\end{equation*}
\end{proof}

\begin{bsp} \label{frac-15}
	Let $(X_t)_{t \geq 0}$ be a stable-like process, i.\,e.\ a L\'evy-type process with symbol $q(x,\xi) = |\xi|^{\alpha(x)}$ for a (continuous) function $\alpha: \mbb{R}^d \to (0,2)$. Then $\beta_0^x \geq \alpha_l := \inf_{y \in \mbb{R}^d} \alpha(y)$ and $\beta_{\infty}^x = \alpha(x)$. Hence, by Corollary~\ref{frac-11}, we have for any $\alpha>\alpha(x)$ and $\kappa  \in [0,\alpha_l)$, \begin{equation*}
		\mbb{E}^x \left(\sup_{s \leq t} |X_s-x|^{\kappa} \right) \leq \begin{cases} C t^{\frac{\kappa}{\alpha} \wedge 1}, & \kappa \neq \alpha, \\ Ct |\log t|, & \kappa = \alpha \end{cases}
	\end{equation*}
	for all $t \leq 1$. Moreover, by Theorem~\ref{frac-13}, 
	\begin{equation*}
			\mbb{E}^x \left( \sup_{s \leq t} |X_s-x|^{\kappa} \right) \leq C t^{\kappa/\beta}
	\end{equation*}
	for all $t \geq 1$ and  any $\beta<\alpha_l$, $\kappa \in [0,\beta)$.
\end{bsp}

For the readers' convenience we sum up conditions and results in Table~\ref{tab1}.

\begin{ack}
	I would like to thank Ren\'e Schilling for his valuable remarks and suggestions. Thanks also to the referees for carefully reading this paper and giving many helpful comments.
\end{ack}

\begin{landscape}
\begin{table} 
	\begin{tabular}{lll|lll}
		assumptions on the symbol & assumptions on the moments & & $\mbb{E}^x \left( \sup_{s \leq t} |X_s-x|^{\kappa} \right)$ & reference  \\
		\hline
		\vtop{ \hbox{\strut ``bounded variation''-type:}  \hbox{\strut $q(x,\xi) = \int_{\mbb{R}^d \backslash \{0\}} (1-e^{i \, y \xi}) \, N(x,dy)$}} & \vtop{ \hbox{\strut $\sup_{x \in \mbb{R}^d} \int_{|y| > 1} |y|^{\alpha} \, N(x,dy)< \infty$} \hbox{\strut $\sup_{x \in \mbb{R}^d} \int_{|y| \leq 1} |y|^{\beta} \, N(x,dy)<\infty$}} & \vtop{ \hbox{\strut $\alpha \in (0,1]$} \hbox{\strut $\beta \in [0,\alpha]$}} & $O(t^{\kappa/\alpha})$ for $\kappa \in [0,\alpha]$ & Theorem~\ref{frac-1} \B\T \\
		\hline
		\vtop{ \hbox{\strut ``bounded variation''-type:}  \hbox{\strut $q(x,\xi) = \int_{\mbb{R}^d \backslash \{0\}} (1-e^{i \, y \xi}) \, N(x,dy)$}} & \vtop{ \hbox{\strut $\sup_{x \in \mbb{R}^d} \int_{|y| > 1} |y|^{\alpha} \, N(x,dy)< \infty$} \hbox{\strut $\sup_{x \in \mbb{R}^d} \int_{|y| \leq 1} |y|^{\beta} \, N(x,dy)<\infty$}} & \vtop{ \hbox{\strut $\alpha \in (0,1]$} \hbox{\strut $\beta \in [\alpha,1]$}} & $O(t^{\kappa/\alpha}+t^{\kappa/\beta})$ for $\kappa \in [0,\alpha]$ & Theorem~\ref{frac-1} \B\T \\
		\hline 
		\vtop{ \hbox{\strut ``pure-jump''-type:}  \hbox{\strut $q(x,\xi) = \int_{\mbb{R}^d \backslash \{0\}} (1-e^{i \, y \xi}+iy \xi \I_{(0,1]}(|y|))) \, N(x,dy)$}} & \vtop{ \hbox{\strut $\sup_{x \in \mbb{R}^d} \int_{|y| > 1} |y|^{\alpha} \, N(x,dy)< \infty$} \hbox{\strut $\sup_{x \in \mbb{R}^d} \int_{|y| \leq 1} |y|^{\beta} \, N(x,dy)<\infty$}} & \vtop{ \hbox{\strut $\alpha \in (0,1]$} \hbox{\strut $\beta \in [1,2]$}} & $O(t^{\kappa/\alpha}+t^{\kappa/\beta})$ for $\kappa \in [0,\alpha]$ & Theorem~\ref{frac-1} \B\T \\
			\hline 
			\vtop{ \hbox{\strut ``martingale''-type:}  \hbox{\strut $q(x,\xi) = \int_{\mbb{R}^d \backslash \{0\}} (1-e^{i \, y \xi}+iy \xi) \, N(x,dy)$}} & \vtop{ \hbox{\strut $\sup_{x \in \mbb{R}^d} \int_{|y| > 1} |y|^{\alpha} \, N(x,dy)< \infty$} \hbox{\strut $\sup_{x \in \mbb{R}^d} \int_{|y| \leq 1} |y|^{\beta} \, N(x,dy)<\infty$}} & \vtop{ \hbox{\strut $\alpha \in [1,2]$} \hbox{\strut $\beta \in [1,2]$}} & $O(t^{\kappa/\alpha}+t^{\kappa/\beta})$ for $\kappa \in [0,\alpha \wedge \beta]$ & Theorem~\ref{frac-5} \B\T \\
			\hline 
			\vtop{ \hbox{\strut ``martingale''-type:}  \hbox{\strut $q(x,\xi) = \int_{\mbb{R}^d \backslash \{0\}} (1-e^{i \, y \xi}+iy \xi) \, N(x,dy)$}} & \vtop{ \hbox{\strut $\sup_{x \in \mbb{R}^d} \int_{|y| > 1} |y|^{\alpha} \, N(x,dy)< \infty$} \hbox{\strut $\sup_{x \in \mbb{R}^d} \int_{|y| \leq 1} |y|^2 \, N(x,dy)<\infty$}} & $\alpha >2$ & $O(t^{\kappa/\alpha}+t^{\kappa/2})$ for $\kappa \in [0,\alpha]$ & Theorem~\ref{frac-5} \B\T \\
			\hline 
			\vtop{\hbox{\strut $\forall |\xi| \leq 1: \sup_{|y-x| \leq |\xi|^{-1}} |q(y,\xi)| \leq C (1+|x|^{\gamma}) |\xi|^{\beta}$} \hbox{\strut $\forall |\xi| \geq 1: \sup_{|y-x| \leq |\xi|^{-1}} |q(y,\xi)| \leq C (1+|x|^{\gamma}) |\xi|^{\alpha}$} } & & $\gamma<\beta$ & \vtop{\hbox{\strut $O(t^{\kappa (\alpha^{-1} \wedge \gamma^{-1})})$, $t \to 0$, for $\kappa \in [0,\beta), \kappa \neq \alpha$} \hbox{\strut}} & Theorem~\ref{frac-7} \B\T \\
			\hline 
			$\forall |\xi| \geq 1: \sup_{|y-x| \leq |\xi|^{-1}} \sup_{|\eta| \leq |\xi|} |q(y,\eta)| \leq C |\xi|^{\alpha}$  & & & $O(t^{\kappa/\alpha \wedge 1})$, $t \to 0$, for $\kappa \in [0,\beta_0^x), \kappa \neq \alpha$ & Corollary~\ref{frac-11} \B\T \\
			\hline
			$\forall |\xi| \leq 1: \sup_{|y-x| \leq |\xi|^{-1}} \sup_{|\eta| \leq |\xi|} |q(y,\eta)| \leq C |\xi|^{\beta}$ & & & $O(t^{\kappa/\beta})$, $t \to \infty$, for $\kappa \in [0,\beta)$ & Theorem~\ref{frac-13} \T
	\end{tabular}
	\caption{Estimates of fractional moments of (pure-jump) L\'evy-type processes.} 
		\label{tab1}
\end{table}
\end{landscape}


\begin{thebibliography}{99}\frenchspacing
	\bibitem{behme}
		Behme, A., Lindner, A.: On exponential functionals of L\'evy processes. \emph{J.\ Theor.\ Probab.} \textbf{28} (2015), 681--720.
	\bibitem{blumen}
		Blumenthal, R., Getoor, R.: Sample functions of stochastic processes with stationary independent increments. \emph{J.\ Math.\ Mech.} \textbf{10} (1961), 493--516.
	\bibitem{ltp}
	    B\"{o}ttcher, B., Schilling, R.L., Wang, J.: \emph{L\'evy-Type Processes: Construction, Approximation and Sample Path Properties}. Springer Lecture Notes in Mathematics vol.\ \textbf{2099}, (vol.~III of the ``L\'evy Matters'' subseries). Springer, Berlin 2014.
	\bibitem{cinlar}
		Cinlar, E., Jacod, J.: Representation of semimartingale Markov processes in terms of Wiener processes and Poisson processes. In: Seminar on Stochastic Processes 1981, pp.~159--242. Birkh\"{a}user, Boston (1981).
	\bibitem{deng}
	   	Deng, C.-S., Schilling, R.\,L.: On shift Harnack inequalities for subordinate semigroups and moment estimates for L\'evy processes. \emph{Stoch.\ Proc.\ Appl.} \textbf{125} (2015), 3851--3878.
	\bibitem{fig}
		Figueroa-L\'opez, J.\,E: Small-time asymptotics for L\'evy processes. \emph{Statistics \& Probability Letters} \textbf{78} (2008), 3355--3365.
	\bibitem{fou}
		Fournier, N., Printems, J.: Absolute continuity for some one-dimensional processes. \emph{Bernoulli} \textbf{16} (2010), 343--360.
	\bibitem{hille}
		Hille, E., Phillips, S.: \emph{Functional analysis and semi groups}. American Math.\ Soc., Providence (RI) 1985. 
	\bibitem{hoh}
		Hoh, W.: \emph{Pseudo-Differential Operators Generating Markov Processes}. Habilitationsschrift. Universit\"{a}t Bielefeld, Bielefeld 1998. https://www.math.uni-bielefeld.de/~hoh/temp/pdo\_mp.pdf.
	\bibitem{jac}
		Jacod, J.: Asymptotic properties of power variations of L\'evy processes. \emph{ESAIM: Probability and Statistics} \textbf{11} (2007), 173--196.
	\bibitem{kallen}
		Kallenberg, O.: \emph{Foundations of Modern Probability Theory}. Springer, 2002 (2nd edition).
	\bibitem{ihke}
		K\"{u}hn, F., Schilling, R.\,L.: On the domain of fractional Laplacians and related generators of Feller processes. Preprint arXiv 1610.08197. 	
     \bibitem{luschgy}
     	Luschgy, H., Pag\`es, G.: Moment estimates for L\'evy Processes. \emph{Elect.\ Comm.\ in Probab.} \textbf{13} (2008), 422--434.
	\bibitem{novi85}
		Novikov, A.\,A.: On discontinuous martingales. \emph{Theory Probab.\ Appl.} \textbf{20} (1975), 11--26.
	\bibitem{sasvari}
	 	Sasv\'ari, Z.: \emph{Multivariate Characteristic and Correlation Functions}. De Gruyter, 2013. 
	 \bibitem{sato}
	   		Sato, K.: \emph{L\'evy Processes and Infinitely Divisible Distributions}. Cambridge University Press, Cambridge 2005.
	 \bibitem{bm2}
	   	   	Schilling, R.\,L., Partzsch, L.: \emph{Brownian Motion: An Introduction to Stochastic Processes}. De Gruyter, Berlin 2014 (2nd edition).
	   \bibitem{rs97}
	   	 Schilling, R.\,L.: Growth and H\"{o}lder conditions for the sample paths of Feller processes. \emph{Probab.\ Theory Relat.\ Fields} \textbf{112} (1998), 565--611.
	   	\bibitem{schnurr}
	   		Schilling, R.\,L., Schnurr, A.: The Symbol Associated with the Solution of a Stochastic Differential Equation. \emph{Electron.\ J.\ Probab.} \textbf{15} (2010), 1369--1393.
\end{thebibliography}
\end{document}